\documentclass{article}

\usepackage{latexsym,amsfonts,amssymb,epsfig,verbatim}
\usepackage{amsmath,amsthm,amssymb,latexsym,graphics,textcomp}
\usepackage{eucal,eufrak}
\usepackage{graphicx}
\usepackage{hyperref}
\usepackage{url}
\input{xy}
\xyoption{all}

\theoremstyle{plain}
\newtheorem{theorem}{Theorem}[section]
\newtheorem{lemma}[theorem]{Lemma}
\newtheorem{corollary}[theorem]{Corollary}

\newtheorem{proposition}[theorem]{Proposition}

\newtheorem*{claim}{Claim}
\newtheorem*{main2}{Theorem \ref{T:main2}}
\newtheorem*{surfacesubgroup}{Corollary \ref{C:surfacesubgroups}}

\def\T{\mathcal{T}}
\def\Mod{\mathrm{Mod}}
\def\WP{\mathrm{WP}}
\def\M{{\widetilde{\mathcal{M}}}}
\def\C{\mathcal{C}}

\def\qi{\raise1.6ex\hbox{\tiny$A$,$B$}\mspace{-23mu}\asymp}

\def\Min{\mathrm{Min}}
\def\syl{\mathrm{syl}}
\def\diam{\mathrm{diam}}
\def\base{\mathrm{base}}

\newcommand{\BF}{\mathbb{F}}
\newcommand{\BX}{\mathbb{X}}
\newcommand{\co}{\colon\thinspace}
\newcommand{\Fill}{\mathrm{Fill}}

\title{The geometry of right angled Artin subgroups of mapping class groups}

\author{Matt T. Clay\thanks{Partially supported by NSF grant
    DMS-1006898.}, Christopher J. Leininger\thanks{Partially supported
    by NSF grant DMS-0905748.}, and Johanna Mangahas\thanks{Partially supported by NSF RTG grant
    0602191.}}

\begin{document}

\maketitle

\begin{abstract}
  We describe sufficient conditions which guarantee that a finite set
  of mapping classes generate a right-angled Artin group
  quasi-isometrically embedded in the mapping class group.  Moreover,
  under these conditions, the orbit map to Teichm\"uller space is a
  quasi-isometric embedding for both of the standard metrics.  As a
  consequence, we produce infinitely many genus $h$ surfaces (for any
  $h$ at least 2) in the moduli space of genus $g$ surfaces (for any
  $g$ at least 3) for which the universal covers are
  quasi-isometrically embedded in the Teichm\"uller space.
\end{abstract}

\section{Introduction}

Let $S$ denote a surface and $\Mod(S)$ its mapping class group. Given
independent pseudo-Anosov mapping classes $f_1,\ldots,f_n \in
\Mod(S)$, McCarthy \cite{mccarthy} and Ivanov \cite{ivanov} proved
that by passing to sufficiently high powers, these mapping classes
generate a free subgroup.  This is the primary ingredient in the proof
that $\Mod(S)$ satisfies the ``Tits alternative''; see also
\cite{fujiwara,mangahas} for quantitative versions of this. Farb and
Mosher \cite{farbmosher} defined a notion of convex cocompactness for
subgroups of $\Mod(S)$ by way of analogy with Kleinian groups, and
proved that $f_1,\ldots,f_n$ could be raised to sufficiently high
powers to further guarantee that the subgroup they generate is convex
cocompact; see also \cite{hypbyhyp,kentleininger,hamenstadt}.

Given an arbitrary set of elements $f_1,\ldots,f_n \in \Mod(S)$, we
cannot expect that they generate a free group upon raising to
sufficiently high powers.  However, Koberda \cite{koberda} has
recently proven that the powers do generate a right-angled Artin
group; see also \cite{crispparis,crispwiest2,crispfarb} for partial
results in this direction.

In this paper, we are interested in geometric properties of
right-angled Artin subgroups of the mapping class group.  As convex
cocompact subgroups are necessarily Gromov hyperbolic, we must
consider other geometric properties for non-free right-angled Artin
subgroups of $\Mod(S)$. For example, Crisp and Wiest
\cite{crispwiest2} produced quasi-isometric embeddings of certain
right-angled Artin groups into braid groups (and hence also mapping
class groups).  In this paper we show that this is possible in much
greater generality, and furthermore, one can often conclude even
stronger geometric statements for the corresponding subgroups.  Here
we state our main theorem, and refer the reader to Section
\ref{S:notation} for necessary terminology and a more precise
statement (Theorem \ref{T:main2}).

\begin{theorem} \label{T:main}
Suppose $f_1,\ldots,f_n \in \Mod(S)$ are fully supported on overlapping nonannular subsurfaces.
Then after raising to sufficiently high powers, these elements generate a quasi-isometrically
embedded right-angled Artin subgroup of $\Mod(S)$.  Furthermore, the orbit map to the Teichm\"uller
space is a quasi-isometric embedding for both of the standard metrics, namely the Teichm\"uller and
Weil--Petersson metrics.
\end{theorem}

\noindent {\bf Remarks.}

\smallskip

\noindent
1. We note that for the second statement to hold, the assumption that the
support of each $f_i$ is not an annulus is necessary.  On the other hand, it seems likely that the
homomorphism to $\Mod(S)$ is a quasi-isometry without this additional assumption.

\smallskip

\noindent
2. There are a number of other ``natural'' metrics on Teichm\"uller space besides the two we have mentioned; the Bergman metric, Carath\'eodory metric, McMullen metric, K\"ahler-Einstein metric, Ricci metric and perturbed Ricci metric.  However, each of these is quasi-isometric to the Teichm\"uller metric (see \cite{mcmullenkahler,yeung,liusunyau1,liusunyau2}), and so the conclusion of Theorem \ref{T:main} also holds for any of these metrics.

\bigskip

In section \ref{S:elements} we use the ideas from the proof of this theorem to describe the
Thurston type of any element in the right-angled Artin subgroup of $\Mod(S)$ we construct, and we
see that it is pseudo-Anosov on the largest possible subsurface.  In particular, we describe
exactly which elements are pseudo-Anosov on $S$; see Theorem \ref{T:find pAs}.

The hypotheses in Theorem \ref{T:main} are general enough to easily provide quasi-isometric
embeddings of any right-angled Artin group into some mapping class group (see the end of Section
\ref{S:realizing a graph}).  In particular we have the following.

\begin{corollary} \label{C:existence}
Any right-angled Artin group admits a homomorphism to some mapping class group which is a
quasi-isometric embedding, and for which the orbit map to Teichm\"uller space is a quasi-isometric
embedding with respect to either of the standard metrics. \qed
\end{corollary}

The fundamental group of a closed orientable surface (of genus $h \geq 2$) is called a {\em (genus
$h$) surface subgroup}.  Many right-angled Artin groups contain quasi-isometrically embedded
surface subgroups; see \cite{servatiusdroms,crispwiest1} (though the question of exactly which
right-angled Artin groups contain surface subgroups is still open; see for example
\cite{gordonlongreid,sanghyunkim1,sanghyunkim2,crispsageevsapir,rover}).  There are also constructions of surface subgroups
of the mapping class group \cite{atiyah,kodaira,gonzalezharvey}.  In \cite{leiningerreid},
infinitely many nonconjugate surface subgroups were constructed with geometric properties akin to
geometric finiteness in the setting of Kleinian groups.  From an explicit version of Corollary
\ref{C:existence}, and the aforementioned examples of surface subgroups of right-angled Artin
groups, we obtain the following.  See Section \ref{S:proofs} for the proof.

\begin{corollary} \label{C:surfacesubgroups}
For any closed surface $S$ of genus at least $3$ and any $h \geq 2$, there exist infinitely many
nonconjugate genus $h$ surface subgroups of $\Mod(S)$, each of which act cocompactly on some
quasi-isometrically embedded hyperbolic plane in the Teichm\"uller space $\T(S)$, with either of
the standard metrics.
\end{corollary}

This corollary is in contrast to the work of Bowditch \cite{bowditch-atoroidal} who proves
finiteness, for any fixed $h \geq 2$, for the number of conjugacy classes of genus $h$ surface
subgroups of $\Mod(S)$ which are purely pseudo-Anosov (we note that surface subgroups of the
mapping class group which arise as subgroups of right-angled Artin groups can never be purely
pseudo-Anosov; see Proposition \ref{P:not purely pa} below).  While these surface subgroups are not
purely pseudo-Anosov, by the corollary, they do have the closely related property that every
nontrivial element has positive translation length on $\T(S)$.

Finally, we remark that while Bowditch's result mentioned above is an example of a kind of rank-$1$
phenomenon for $\Mod(S)$, our examples illustrate higher rank behavior.  Specifically, we could
compare our results with those of Wang \cite{wang}, who finds infinitely many conjugacy classes of
discrete, faithful representations of right-angled Artin groups (hence surface subgroups) into
higher rank Lie groups.  Furthermore, Long, Reid and  Thistlethwaite \cite{longreidthistlethwaite},
find infinitely many conjugacy classes of Zariski dense, purely semi-simple representations of a
surface group into $SL(3,\mathbb Z)$.  In fact, these surface groups are very closely related to
the ones we study, in the sense that every nontrivial element has positive translation length on the
associated symmetric space.

\subsection{Plan of the paper}

We begin in Section \ref{S:notation} by setting up the relevant definitions and notation we will
use throughout.  The section ends with a more precise version of our main theorem (Theorem
\ref{T:main2}).  In Section \ref{S:projections and distance} we describe an alternative space on
which $\Mod(S)$ acts, namely Masur and Minsky's graph of {\em markings} \cite{mm2}.  We also state
the required {\em distance formulas} (Theorems \ref{T:mod distance}, \ref{T:WP distance} and
\ref{T:Teich distance}) which provide the coarse estimates for the distances in the desired spaces,
$\Mod(S)$ and $\T(S)$, in terms of sums of ``local distances'' between pairs of markings. These
local distances are precisely the {\em subsurface distances}, also described in this section.

The idea of the proof of Theorem \ref{T:main2} is as follows.  The hypothesis implies that each of
the generators of the right-angled Artin group corresponds to a mapping class which makes progress
in some subsurface---that is, it contributes nontrivially to some local distance. A geodesic in the
Cayley graph of the right-angled Artin group determines a sequence of mapping classes, each of
which makes progress in some subsurface.  We need only ensure that this progress accumulates (that
is, we need to avoid cancelation of local distances).  This is verified by Theorem \ref{T:main
technical}, which relates a partial order on the set of syllables in a minimal length
representative for an element of the right-angled Artin group (see Section \ref{S:normal forms})
with the partial order from \cite{bkmm} on the set of subsurfaces ``between'' a marking and its
image under the associated mapping class (see Section \ref{S:time order}).  The details of the
proof of Theorem \ref{T:main technical} are carried out in Section \ref{S:proofs}, followed by the
proof of Theorem \ref{T:main2}.

In Section \ref{S:elements} we find the Thurston type of each element in the right-angled Artin
subgroups of $\Mod(S)$ we are considering.  We show that by conjugating to use the minimal number
of generators to represent the element, it will be pseudo-Anosov on the smallest subsurface filled
by the supports of the generators.  For this, we use Masur and Minsky's {\em Bounded Geodesic Image Theorem} \cite{mm2} to prove that the element acts with positive translation distance on the curve complex of the appropriate subsurface.

We end with a discussion of surface subgroups and the proofs of Corollary \ref{C:surfacesubgroups} and Proposition \ref{P:not purely pa}.\\

\noindent {\bf Acknowledgements.}  We would like to thank Richard Kent,  Alan Reid, Thomas Koberda and Jason Behrstock
for helpful conversations, and the {\em Hausdorff Research Institute for Mathematics} in Bonn, Germany, for its
hospitality while this work was being completed.

\section{Notation and terminology} \label{S:notation}

\subsection{Quasi-isometries}

Given $A \geq 1$ and $B \geq 0$, we write $x \qi y$ to mean
\[ \frac{y- B}{A} \leq x \leq Ay+B \]

If $(\frak X_1,d_1)$ and $(\frak X_2,d_2)$ are metric spaces and $A \geq 1$, $B \geq 0$, then an
{\em $(A,B)$--quasi-isometric embedding} from $\frak X_1$ to $\frak X_2$ is a map
\[F:\frak X_1 \to \frak X_2 \]
with the property that for all $x,y \in \frak X_1$, we have
\[ d_1(x,y) \qi d_2(F(x),F(y)).\]
If $F$ is an $(A,B)$--quasi-isometric embedding for some $A$ and $B$, then we will say that $F$ is
a {\em quasi-isometric embedding}.

If $F:\frak X_1 \to \frak X_2$ is a quasi-isometric embedding and there is a constant $D > 0$ so
that any point of $\frak X_2$ is within $D$ of some point of $F(\frak X_1)$, then $F$ is called a
{\em quasi-isometry}.

\subsection{Right angled Artin groups} \label{S:raag}

Let $\Gamma$ be a graph with vertex set $\{s_1,\ldots,s_n\}$.  The associated {\em right-angled
Artin group} $G = G(\Gamma)$, is defined to be the group with presentation
\[ G = \langle s_1,\ldots,s_n \, | \, [s_i,s_j] = 1 \mbox{ if }\{s_i,s_j\}\mbox{ is an edge of }\Gamma \rangle. \]
We will always work with the word metric on $G$ with respect to this generating set, and will
denote it $d_G$.

Examples of right-angled Artin groups are free groups, and direct products of free groups (in
particular, free abelian groups).  A simple example of a right-angled Artin group which is neither
free nor a product of free groups is $G(\Gamma)$ where $\Gamma$ is the cyclic graph with $5$
vertices shown in Figure \ref{F:pentagon}.

\begin{figure}[htb]
\begin{center}
\includegraphics[height=1.5truein]{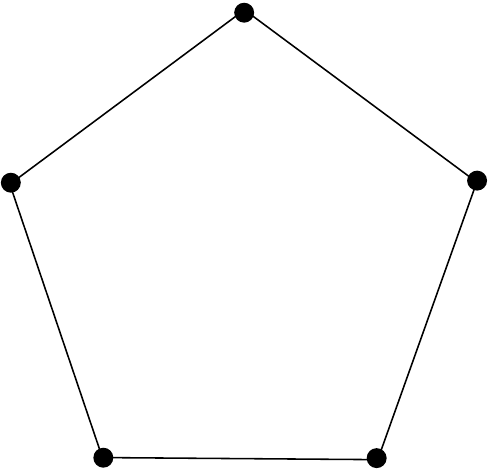}
\caption{The cyclic graph with $5$ vertices.} \label{F:pentagon}
\end{center}
\end{figure}

\subsection{Surfaces}

Given a connected surface $S$ of genus $g$ with $n$ punctures, the
{\em complexity} is defined to be $\xi(S)=3g-3+n$.  Unless otherwise
stated, we will assume throughout that $\xi(S) > 0$.  The {\em mapping
  class group} of $S$ is the group of isotopy classes of orientation
preserving homeomorphisms of $S$ and is denoted $\Mod(S)$. By a {\em
  curve} in $S$, we mean the isotopy class of an essential
(non-null-homotopic and non-peripheral) simple closed curve. A {\em
  pants decomposition of $S$} is a maximal collection of pairwise
disjoint curves in $S$.  Since $\xi(S) > 0$, a nonempty pants
decomposition exists and has precisely $\xi(S)$ curves in it.

A subsurface $X \subset S$ is {\em essential} if it is either a regular neighborhood of an
essential simple closed curve, or else a component of the complement of an open regular
neighborhood of a (possibly empty) union of pairwise disjoint essential simple closed curves. In
particular, we assume that essential subsurfaces are connected. We will generally not distinguish
between punctures and boundary components, and if $X \subset S$ has genus $h$ with $k$ punctures
and $b$ boundary components, then we will write $\xi(X) = 3g-3 + k + b$.  Finally, we will assume
that an essential subsurface $X$ has $\xi(X) \neq 0$, thus excluding a pair of pants as an
essential subsurface. The set of all isotopy classes of essential subsurfaces $X$ of $S$ with
$\xi(X) \neq 0$ will be denoted $\Omega(S)$.

\begin{figure}[htb]
\begin{center}
\includegraphics[height=1.5truein]{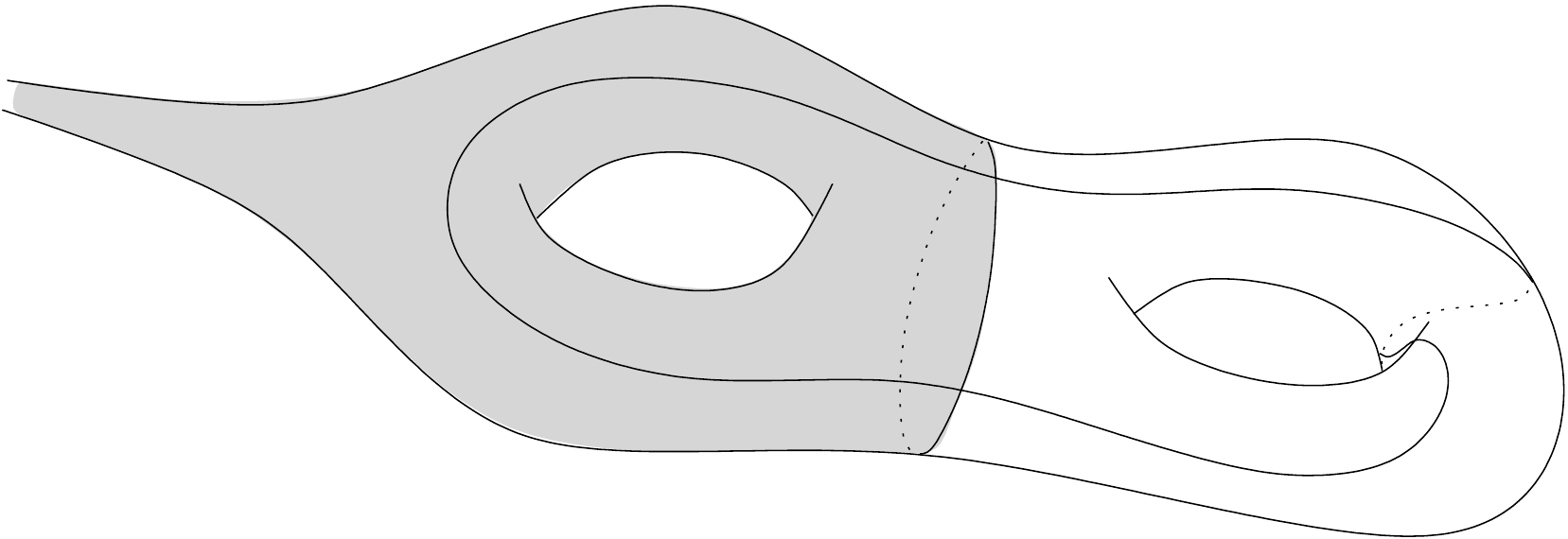}
\caption{A genus $2$ surface with $1$ puncture $S$, a subsurface $X$ (shaded) and a curve
$\gamma$.} \label{F:subsurface}
\end{center}
\end{figure}

We will often refer to the isotopy class of an essential subsurface simply as a {\em subsurface}.
Furthermore, we will choose nice representative for each curve and each subsurface, and will not
distinguish between a representative and its isotopy class when it is convenient.  To be precise,
we choose representatives as follows (annuli will play essentially no role in our discussion, so we
do not bother describing their preferred representatives).

Fix a complete hyperbolic metric on $S$, and realize each curve by its unique geodesic
representative. These representatives minimize the number of intersections (that is, they realize
geometric intersection number).  For each curve $\alpha$, we may choose some
$\epsilon_\alpha$--neighborhood $N(\alpha)$ so that for any curves $\alpha$ and $\beta$, the
intersections of $N(\alpha)$ and $N(\beta)$ correspond precisely to the intersections of $\alpha$
and $\beta$, and each such intersection is a ``product square'' (see Figure \ref{F:realization}).
For any nonannular subsurface $X$, which is a component of the complement of an open regular
neighborhood of $\alpha_1 \cup \cdots \cup \alpha_k$, we take its representative to be defined
as the corresponding component of the complement of the interior of the neighborhood $N(\alpha_1)
\cup \cdots \cup N(\alpha_k)$.

Suppose $X,Y \subsetneq S$ are representative subsurfaces.  Observe that $X \cap Y = \emptyset$ if
and only if $X$ and $Y$ cannot be isotoped to be disjoint.  In this case, we say that $X$ and $Y$
are {\em overlapping}, and write $X \pitchfork Y$ if $X \not \subseteq Y$ and $Y \not \subseteq X$.
One can check that this notion of overlapping agrees with that defined in \cite{bkmm}, which is to
say that $X \pitchfork Y$ if and only if some component of $\partial X$ cannot be isotoped disjoint
from $Y$ and some component of $\partial Y$ cannot be isotoped disjoint from $X$.

\subsection{Realizing a graph} \label{S:realizing a graph}

Given a graph $\Gamma$, a surface $S$, and a collection of nonannular subsurfaces $X_1,\ldots,X_n
\subset S$, we say that $\mathbb X = \{X_1,\ldots,X_n\}$ {\em realizes $\Gamma$ nicely in $S$} if
\begin{enumerate}
\item[(1)] $X_i \cap X_j = \emptyset$ if and only if $\{s_i,s_j\}$ is an edge of $\Gamma$, and
\item[(2)] whenever $X_i \cap X_j \neq \emptyset$, then $X_i \pitchfork X_j$.
\end{enumerate}


As Figure \ref{F:genus08cyclic5} indicates, there is a nice realization of the cyclic graph of
length $5$ in a genus $3$ surface obtained from a branched cover of the sphere, branched over 8
points.   By adding more points to this picture and taking a branched cover, we can produce nice
realizations of this graph in any surface of genus $g \geq 3$.  Moreover, given any graph it is
easy to find some surface and a collection of subsurfaces which provide a nice realization (see
\cite{crispparis,crispwiest1} for this kind of construction). We sketch one such construction here.

\begin{figure}[htp]
\begin{center}
\includegraphics[height=2truein]{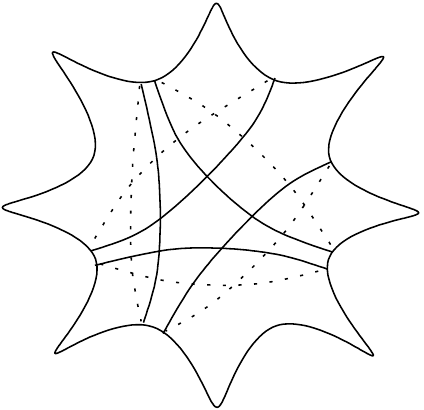}
\caption{This figure represents a sphere with 8 punctures containing five curves, each of which
bounds a disk with 3 punctures.  These five $3$--punctured disks provide a nonannular realization
of the cyclic graph with $5$ vertices.  Taking a two-fold branched cover over the 8 points, we
obtain a nonannular realization on a genus $3$ surface (by $1$--holed tori).}
\label{F:genus08cyclic5}
\end{center}
\end{figure}

\begin{figure}[htp]
\begin{center}
\includegraphics[height=3truein]{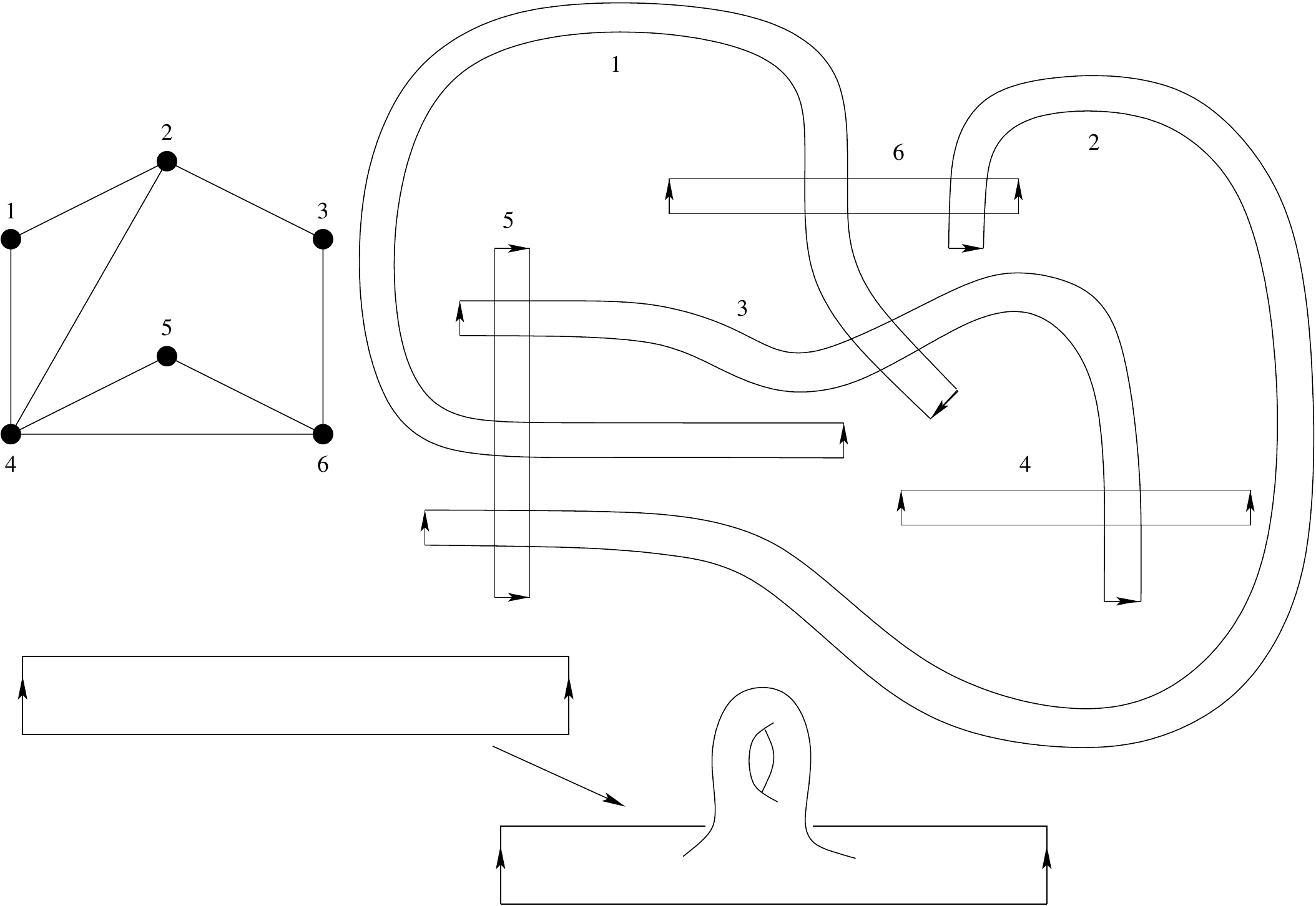}
\caption{A graph $\Gamma$ and the associated annuli glued together along product squares as
prescribed by $\Gamma$. At the bottom, we glue in a $1$--holed torus to an annulus with a disk
removed.} \label{F:realization}
\end{center}
\end{figure}

Starting with a graph $\Gamma$, we take a disjoint union of annuli, one for each vertex of
$\Gamma$.  Next, glue together the annuli along product squares whenever the associated vertices of
$\Gamma$ are not connected by an edge.  In each annulus, remove a disk and glue in a $1$--holed
torus. Finally, cap off the boundary components of the resulting surface with disks. See Figure
\ref{F:realization} for a particular example.

If $X$ is a nonannular subsurface of $S$ and $f \in \Mod(S)$ is the identity outside $X$, we say
that $f$ is {\em supported on} $X$.  We say that $f$ is {\em fully supported on $X$} if we also
have that $f$ is pseudo-Anosov on $X$.  If $f$ is supported on $X$, then $f$ acts on $\C(X)$, the
{\em curve complex of $X$}, and we let $\tau_X(f)$ denote the translation length of $f$ on $\C(X)$.
By a theorem of Masur and Minsky \cite{mm1}, if $f$ is supported on $X$, then it is fully supported
on $X$ if and only if $\tau_X(f) > 0$.  We refer the reader to \cite{mm1} for more details.

\subsection{Homomorphisms} \label{S:homomorphisms}

Suppose now that $\mathbb X = \{X_1,\cdots,X_n\}$ nicely realizes $\Gamma$ in $S$ and that $\mathbb
F = \{f_1,\ldots,f_n\} \subset \Mod(S)$ are mapping classes.  We say that $\mathbb F$ is (fully)
supported on $\mathbb X$ if $f_i$ is (fully) supported on $X_i$ for each $i = 1,\ldots,n$.  Since
homeomorphisms on disjoint subsurfaces commute, there is a unique homomorphism
\[\phi_\mathbb F:G \to \Mod(S) \]
defined by $\phi_\mathbb F(s_i) = f_i$.

We now state a more precise version of our main theorem.  We write $\T(S)$ for the Teichm\"uller
space, and we denote its two standard metrics by $d_\T$ for the Teichm\"uller metric and $d_\WP$
for the Weil--Petersson metric.

\begin{theorem} \label{T:main2} Given a graph $\Gamma$ and a nice realization $\mathbb X =
\{X_1,\ldots,X_n\}$ of $\Gamma$ in $S$, there exists a constant $C > 0$ with the following
property. If $\mathbb F = \{f_1,\ldots,f_n\}$ is fully supported on $\mathbb X$ and
$\tau_{X_i}(f_i) \geq C$ for all $i = 1,\ldots,n$, then the associated homomorphism
\[ \phi_\mathbb F:G(\Gamma) \to \Mod(S) \]
is a quasi-isometric embedding.  Furthermore, the orbit map $G \to \T(S)$ is a quasi-isometric
embedding for both $d_\T$ and $d_\WP$.
\end{theorem}

\noindent {\bf Remark.} We reiterate for the casual reader that the subsurfaces $X_i$ are
assumed to be essential, connected, and nonannular.\\\\
The proof of Theorem \ref{T:main2} will be carried out in Section \ref{S:proofs}.  Theorem
\ref{T:main2} easily implies Theorem \ref{T:main}.

\section{Projections and distance estimates} \label{S:projections and distance}

Our proof of Theorem \ref{T:main2} uses results from \cite{mm2}, \cite{brock}, \cite{rafi-metric}
and \cite{bkmm}.  The main construction we will use is that of {\em subsurface projection}, which
we now briefly recall.

\subsection{Projections} \label{S:projections}

\begin{figure}[ht]
\begin{center}
\includegraphics[height=1.5truein]{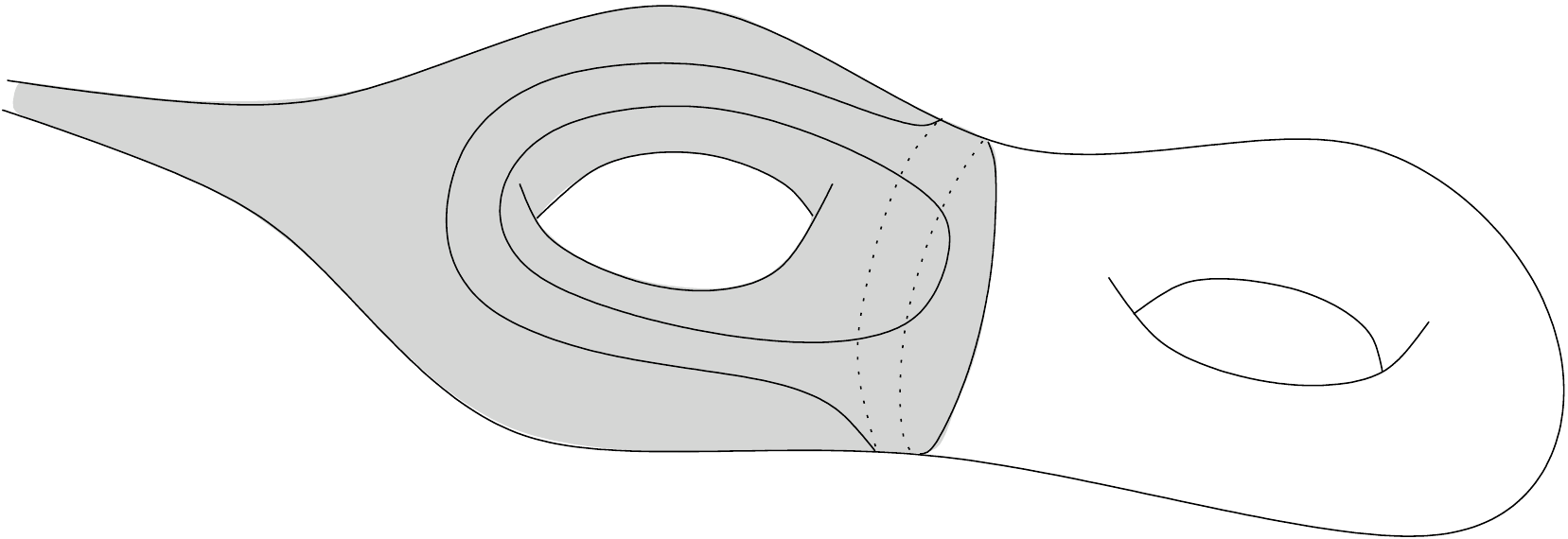}
\caption{The projection $\pi_X(\gamma)$, where $S$, $X$ and $\gamma$ are as in Figure
\ref{F:subsurface}.} \label{F:projection}
\end{center}
\end{figure}

Given a nonannular subsurface $X$ of $S$ and a curve $\gamma$, we define the {\em projection of
$\gamma$ to $X$}, denoted $\pi_X(\gamma)$, to be the subset of $\C(X)$ constructed as follows. If
$\gamma \cap X \neq \emptyset$, then either $\gamma$ is an essential simple closed curve in $X$,
and we define $\pi_X(\gamma) = \{\gamma\}$, or else $\gamma \cap X$ is a disjoint union of
essential arcs in $X$.  For each arc, consider $N$, the regular neighborhood of the arc union the
boundary components of $X$ which the arc meets. Then the boundary of $N$ is a union of curves in
$X$ (and components of $\partial X$), and we define $\pi_X(\gamma)$ to be the set of all such
curves in $X$, over all arcs of $\gamma \cap X$.  See Figure \ref{F:projection}. In general, the
curves in $\pi_X(\gamma)$ need not be disjoint, but the set has diameter at most $2$; see
\cite{mm2}.

When $X$ is an annulus and $\gamma$ a curve, there is also a notion of a projection to $X$, which
assigns to $\gamma$ a diameter one subset of the arc complex of $X$, denoted $\C(X)$, and again we
denote this by $\pi_X(\gamma)$.  For our purposes, simply the existence of this projection will
suffice, so for the details of its definition, we refer the reader to \cite{mm2}.

If $\gamma$ is a disjoint union of curves $\gamma_1 \cup \cdots \cup \gamma_k$, then we define
$\pi_X(\gamma)$ to be the union $\bigcup_i \pi_X(\gamma_i)$.  This set also has diameter at most $2$.
If $\gamma \cap X = \emptyset$, then $\pi_X(\gamma) = \emptyset$.

\subsection{Markings}

Another object we will need is a {\em marking}.  For us, this will mean a {\em complete clean
marking} in the sense of Masur and Minsky \cite{mm2}.  More precisely, a marking $\mu$ is a pants
decomposition called the {\em base of $\mu$}
\[ \base(\mu) = \{\alpha_1,\ldots,\alpha_{\xi(S)}\},\]
together with a {\em transversal} for each curve $\alpha_i \in \base(\mu)$: this is a diameter at
most one subset of $\C(X_i)$, where $X_i$ is the annular neighborhood of $\alpha_i$, together with some
additional properties which we will not need descriptions for; see \cite{mm2} for a discussion.

Masur and Minsky \cite{mm2} identify the set of all markings with the vertex set of a graph $\M(S)$
called the {\em marking graph of $S$}.  The edges of this graph correspond to certain {\em
elementary moves} one can perform on a marking.  We denote the resulting path metric on $\M(S)$ by
$d_\M$. The graph $\M(S)$ is locally finite, and $\Mod(S)$ acts by isometries on it.  In
particular, the orbit map of this action is a quasi-isometry.  We will use $\M(S)$ as a model for
$\Mod(S)$.

Any marking $\mu$ can be projected to a subsurface.   If $X$ is a nonannular subsurface, then
$\pi_X(\mu)$ is defined to be $\pi_X(\base(\mu))$.  For annuli, the projection is defined differently;
see \cite{mm2}.

\subsection{Distances}\label{S:distances}

Given a subsurface $X$ and curves or markings $\mu$ and $\mu'$, we
define their {\em distance in $X$} to be
\[ d_X(\mu,\mu') = \diam (\pi_X(\mu) \cup \pi_X(\mu')) \]
where the diameter is computed in $\C(X)$.

A trivial observation is that if $\mu,\mu'$ are curves or markings on
$S$, $f \in \Mod(S)$ is supported on $X$, and $Y$ is a nonannular
subsurface disjoint from $X$ such that $\mu$ and $\mu'$ have nonempty
projection to $Y$ then
\[ d_Y(\mu,f(\mu')) = d_Y(\mu,\mu'). \]

\noindent {\bf Remark.} We note that the validity of this observation relies on the assumption that
$Y$ is nonannular.\\

Given $K > 0$ and $\mu,\mu' \in \M(S)$, define
\[ \Omega(K,\mu,\mu') = \{ X \subseteq S \, | \, \xi(X) \geq 1 \mbox{ or } X \mbox{ is an annulus, and } d_X(\mu,\mu') \geq K \}. \]
It is convenient to decompose $\Omega(K,\mu,\mu')$ into the annular subsurfaces
$\Omega_a(K,\mu,\mu')$ and the nonannular subsurfaces $\Omega_n(K,\mu,\mu')$.

The following theorem is proven in \cite{mm2}.

\begin{theorem}[Masur-Minsky] \label{T:mod distance}
There exists $K_0 > 0$ (depending on $S$) so that if $K \geq K_0$, then there exists $A\geq 1$,
$B\geq 0$ with the following property. Given $\mu,\mu' \in \M(S)$ then
\[ d_\M(\mu,\mu') \, \, \qi \sum_{X \in \Omega(K,\mu,\mu')} d_X(\mu,\mu')\]
\end{theorem}

A theorem of Brock \cite{brock} states that the Weil--Petersson metric on Teichm\"uller space is
quasi-isometric to the {\em pants graph}.  In \cite{mm2}, Masur and Minsky give a formula similar
to that of the previous formula for distance in the pants graph.  In particular combining these two
results one obtains the following.

\begin{theorem}[Brock, Masur-Minsky] \label{T:WP distance}
There exists $K_0 > 0$ (depending on $S$) so that if $K \geq K_0$, then there exists $A\geq 1$,
$B\geq 0$ with the following property. If $\mu,\mu' \in \M(S)$ are shortest markings for $m,m' \in
\T(S)$, respectively, then
\[ d_\WP(m,m') \, \, \qi \sum_{X \in \Omega_n(K,\mu,\mu')} d_X(\mu,\mu')\]
\end{theorem}
A {\em shortest marking} for $m$ is just a marking for which the pants decomposition has the
shortest total length among all pants decompositions, and the transversals are projections of the
shortest curves among those which can be used for transversals.   For this theorem, the
transversals are unimportant.

The analogous result for the Teichm\"uller metric was proven by Rafi in \cite{rafi-metric}.

\begin{theorem}[Rafi] \label{T:Teich distance}
There exists $K_0 > 0$ (depending on $S$) so that if $K \geq K_0$ and $\epsilon > 0$ then there
exists $A\geq 1$, $B\geq 0$ with the following property. If $\mu,\mu' \in \M(S)$ are shortest
markings for $m,m'$ in the $\epsilon$--thick part of $\T(S)$, respectively, then
\[ d_\T(m,m') \, \,  \qi \sum_{X \in \Omega_n(K,\mu,\mu')} d_X(\mu,\mu') + \sum_{X \in \Omega_a(K,\mu,\mu')} \log(d_X(\mu,\mu')) \]
\end{theorem}

\noindent {\bf Remark.}  Strictly speaking, Theorems \ref{T:mod distance} and \ref{T:WP distance}
would be sufficient for our purposes since, up to a constant, $d_\WP$ provides a lower bound for
$d_\T$ by a result of Linch \cite{linch}, and the lower bound on distortion is the only nontrivial
inequality we need to prove. However, it seems worthwhile to include Theorem \ref{T:Teich distance}
as this illustrates a common interpretation for all of the metric spaces $\M(S)$ (or $\Mod(S)$),
$(\T(S),d_\T)$, and $(\T(S),d_\WP)$.\\

One final result about distances and subsurface projections which we will need is the following Bounded Geodesic Image Theorem \cite{mm2}.

\begin{theorem}[Masur-Minsky] \label{T:BGI}
There exists $K_0 > 0$ (depending on $S$) so that if $\{v_1,\ldots,v_n\}$ is a geodesic in $\C(S)$ and $X \in \Omega(S)$, then either $\pi_Y(v_j) = \emptyset$ for some $j$ or else
\[ \diam_X(\{\pi_X(v_1),\ldots,\pi_X(v_n)\}) < K_0.\]
\end{theorem}

In particular, note that if $v,v' \in \C(S)$ are two curves with $d_X(v,v') \geq K_0$, then any geodesic between $v$ and $v'$ in $\C(S)$ must pass through a curve $v''$ disjoint from $X$ (for example, it may pass through a curve in $\partial X$).

For simplicity, we will assume, as we may, that $K_0$ is the same constant in all of the theorems in this section.\\

\subsection{Partial order on subsurfaces} \label{S:time order}

In \cite{bkmm}, Behrstock, Kleiner, Minsky and Mosher defined a partial order on
$\Omega(K,\mu,\mu')$ (for $K$ sufficiently large) which is closely related to the time-order
constructed in \cite{mm2} (see also \cite{bm}). However, as is noted in \cite{bkmm}, while the
time-order in \cite{mm2} (which is defined on geodesics in hierarchies) requires a fair amount of
the hierarchy machinery to describe it, the partial order on $\Omega(K,\mu,\mu')$ is completely
elementary.   As this is the basic tool we will use, we include the construction and verification
of the necessary properties of this partial order, for the sake of completeness.

The starting point is the ``Behrstock inequality'' \cite{behr} (see also \cite{mangahas}, Lemma
2.5, for the version stated here).
\begin{proposition} \label{P:Behrstock}
Suppose $X$ and $Y$ are overlapping subsurfaces of $S$ and $\mu$ is a marking on $S$.  Then
\[ d_X(\partial Y,\mu) \geq 10 \Rightarrow d_Y(\partial X,\mu) \leq 4. \]
\end{proposition}

Suppose $K \geq 20$ and we define the partial order as follows. Given $X,Y \in \Omega(K,\mu,\mu')$
with $X \pitchfork Y$, then we write $X \prec Y$ if
\begin{equation} \label{E:partialorder} d_X(\mu,\partial Y) \geq 10. \end{equation}
That this is a strict partial order is a consequence of the following useful description of
$\prec$.

\begin{proposition} \label{P:orderdescribed}
Suppose $K \geq 20$ and $X,Y \in \Omega(K,\mu,\mu')$ with $X \pitchfork Y$.  Then $X$ and $Y$ are
ordered and the following are equivalent
\[\begin{array}{ll}
(1) \quad X \prec Y \quad \quad \quad \quad \quad \quad \quad & (5) \quad d_Y(\mu',\partial X) \geq 10 \quad \quad \quad \quad \quad \quad\\\\
(2) \quad d_X(\mu,\partial Y) \geq 10 & (6) \quad d_Y(\mu',\partial X) \geq K - 4\\\\
(3) \quad d_X(\mu,\partial Y) \geq K - 4 & (7) \quad d_Y(\mu,\partial X) \leq 4\\\\
(4) \quad d_X(\mu',\partial Y) \leq 4\\ \end{array} \]
\end{proposition}
\begin{proof}
Assume the hypothesis of the proposition.  Since $X \pitchfork Y$, we know that $\pi_X(\partial Y)
\neq \emptyset$ and $\pi_Y(\partial X) \neq \emptyset$. To verify the equivalences, first observe
that (1) and (2) are equivalent by definition, and since $K - 4 > 10$, (3) implies (2) and (6)
implies (5). Next, since $d_X(\mu,\mu'),d_Y(\mu,\mu') \geq K$, the triangle inequality guarantees
that (4) implies (3) and (7) implies (6). Furthermore, since $K-4 > 10$, Proposition
\ref{P:Behrstock} tells us that (2) implies (7) and (5) implies (4). This proves all the required
implications.

Finally, we prove that $X$ and $Y$ are ordered.  By the triangle inequality we have
\[ 20 \leq K \leq d_X(\mu,\mu') \leq d_X(\mu,\partial Y) + d_X(\mu',\partial Y).\]
and so one of $d_X(\mu,\partial Y)$ or $d_X(\mu',\partial Y)$ is at least 10.  If $d_X(\mu,\partial
Y) \geq 10$ then $X \prec Y$.  If $d_X(\mu',\partial Y) \geq 10$, then reversing the roles of $X$
and $Y$ in each of the 7 equivalent statements we see that $Y \prec X$, as required.
\end{proof}

\begin{corollary}\label{C:partialorder}
Suppose $K \geq 20$.  Then the relation $\prec$ is a strict partial order.
\end{corollary}
\begin{proof}
Since we never have $X \pitchfork X$, it follows that $\prec$ is non-reflexive.  Furthermore, the
equivalence of (2) and (7) in Proposition \ref{P:orderdescribed} means that $X \prec Y$ implies $Y
\not \prec X$, so $\prec$ is antisymmetric.  Finally, if $X \prec Y$ and $Y \prec Z$ then we know
$\pi_Y(\partial X)$ and $\pi_Y(\partial Z)$ are nonempty, and appealing to Proposition
\ref{P:orderdescribed} and the triangle inequality we have
\[ 20 \leq K \leq d_Y(\mu,\mu') \leq d_Y(\mu,\partial X) + d_Y(\partial X,\partial Z) + d_Y(\mu',\partial Z) \leq d_Y(\partial X,\partial Z) + 8\]
and so
\[ d_Y(\partial X,\partial Z) \geq 12 > 10.\]
In this case, $\partial X$ and $\partial Z$ intersect nontrivially in $Y$, so in particular, $X
\pitchfork Z$.

Now we apply Proposition \ref{P:Behrstock} to the preceding inequality to obtain
\[ d_X(\partial Y,\partial Z) \leq 4 \]
and hence by the triangle inequality
\[ 16 \leq K - 4 \leq d_X(\partial Y,\mu) \leq d_X(\partial Y,\partial Z) + d_X(\partial Z,\mu) \leq 4 + d_X(\partial Z,\mu).\]
Therefore, $d_X(\partial Z,\mu) \geq 12 > 10$, and $X \prec Z$.
\end{proof}


\section{Normal forms in right-angled Artin groups} \label{S:normal forms}


Here we describe the normal forms in $G = G(\Gamma)$ as defined by Green \cite{green}, and
Hermiller and Meier's procedure for obtaining these normal forms \cite{hermiller-meier}.  We refer
the reader to Charney's survey article \cite{charney} for a discussion.

Suppose $w = x_1^{e_1} \cdots x_k^{e_k}$ is a word in the generators:  $x_i \in \{s_1,\ldots,s_n\}$
and $e_i \in \mathbb Z$. Each $x_i^{e_i}$ is called a {\em syllable of $w$}.  We consider the
following moves which can be applied to $w$ (see also \cite{hsu-wise}):
\begin{enumerate}
\item Remove a syllable $x_i^{e_i}$ if $e_i =0$.
\item If $x_i = x_{i+1}$, then replace consecutive syllables $x_i^{e_i} x_{i+1}^{e_{i+1}}$ by $x_i^{e_i+e_{i+1}}$.
\item If $[x_i,x_{i+1}] = 1$, then replace $x_i^{e_i} x_{i+1}^{e_{i+1}}$ with $x_{i+1}^{e_{i+1}} x_i^{e_i}$.
\end{enumerate}

Let $\Min(\sigma)$ be the set of words representing $\sigma \in G$ with the fewest number of
syllables.  Green's normal form for $\sigma$ is a certain type of element of $\Min(\sigma)$
obtained by stringing together, from left to right, maximal collections of commuting syllables. For
us, we will consider any element of $\Min(\sigma)$ as a normal form, and we will shortly impose
some additional structure on the set of syllables.  First, we state the following from
\cite{hermiller-meier}.

\begin{theorem}[Hermiller-Meier]
Any word representing $\sigma \in G$ can be transformed to any element of $\Min(\sigma)$ by
applying a sequence of the moves above. In particular, in any such sequence, the number of
syllables and the length does not increase.
\end{theorem}
It follows that the words in $\Min(\sigma)$ determine geodesics in (the Cayley graph of) $G$ with
respect to $s_1,\ldots,s_n$. Moreover, note that any two elements of $\Min(\sigma)$ differ by
moves of type (3).\\

Let $w = x_1^{e_1} \cdots x_k^{e_k} \in \Min(\sigma)$ and consider the set of syllables $\syl(w) =
\{x_i^{e_i}\}_{i=1}^k$.  We consider this as a set of $k$ distinct elements:  for example, we can
artificially write this as $\{ (x_i^{e_i},i) \}_{i=1}^k$.  If we have two elements $w,w' \in
\Min(\sigma)$ that differ by a single application of move (3) above, then there is an obvious
bijection between $\syl(w)$ and $\syl(w')$.  Moreover, any sequence of these types of moves results
in a sequence of bijections between the syllables of consecutive words in $\Min(\sigma)$.  Observe
that any such bijection between $\syl(w)$ and $\syl(w')$ sends a syllable of $w$ to one of $w'$
representing the same element of $G$.

From this it follows that if any such sequence of moves ever brings a word $w$ back to itself, then
the bijection from $\syl(w)$ to itself is the identity.  The reason is that if $x_i^{e_i}$ and
$x_j^{e_j}$ are syllables of $w$ which represent the same element in $G$ (so $x_i = x_j$, $e_i =
e_j$), then if $x_i^{e_i}$ precedes $x_j^{e_j}$ in $w$, any of the bijections will preserve this
property: a sequence of type (3) moves which would theoretically accomplish a swap of their
positions making $x_j^{e_j}$ precede $x_i^{e_i}$ would require a move where $x_i^{e_i}$ and
$x_j^{e_j}$ are adjacent, at which time a type (2) move could be applied to reduce the number of syllables, and
this is impossible. We use these bijections to identify the syllables of any two words $w,w' \in
\Min(\sigma)$, and simply write $\syl(\sigma)$ for this set of syllables.

We can define a strict partial order on this set of syllables, denoted $\syl(\sigma)$, by declaring
$x_i^{e_i} \prec x_j^{e_j}$ if and only if $x_i^{e_i}$ precedes $x_j^{e_j}$ in {\em every} word $w
\in \Min(\sigma)$.  So for any $w \in \Min(\sigma)$, the order of the syllables is a refinement of
the partial order (and the partial order is the largest partial order having this property for
every $w \in \Min(\sigma)$).

\section{\texorpdfstring{The proof of Theorem \ref{T:main2}.}{The
    proof of Theorem 2.1}} \label{S:proofs}

Throughout this section, we will assume $\mathbb X = \{X_1,\ldots,X_n\}$ realizes $\Gamma$ nicely
in $S$, $\mathbb F =\{f_1,\ldots,f_n\}$ is fully supported on $\mathbb X$, and $\phi_\mathbb
F:G=G(\Gamma) \to \Mod(S)$ is the associated homomorphism.

Given a word $x_1^{e_1}\cdots x_k^{e_k}$ with $x_i \in \{s_1,\ldots,s_n\}$ for all $i$, let $J(i)
\in \{1,\ldots,n\}$ be the unique number for which $x_i = s_{J(i)}$. For any $\sigma \in G$ and $w
= x_1^{e_1}\cdots x_k^{e_k} \in \Min(\sigma)$, set
\[X^w(x_i^{e_i}) = \phi_\mathbb F(x_1^{e_1}\cdots x_{i-1}^{e_{i-1}})(X_{J(i)}) \]
for $i = 2,\ldots,k$ and define $X^w(x_1^{e_1}) = X_{J(1)}$. We think of this as defining a map
\[X^w:\syl(\sigma) \to \Omega(S).\]

\begin{lemma}\label{L:Xsigma}
  Suppose $\Gamma$, $\mathbb X$ and $\mathbb F$ are as above.  If
  $\sigma \in G(\Gamma)$ and $w,w' \in \Min(\sigma)$, then $X^w =
  X^{w'}:\syl(\sigma) \to \Omega(S)$.
\end{lemma}

\begin{proof}
Since any two words $w,w' \in \Min(\sigma)$ differ by a sequence of moves of type (3), that is, in
which adjacent commuting syllables are exchanged, it suffices to verify the lemma in the case that
$w$ and $w'$ differ by such a move:
\[ w = x_1^{e_1}\cdots x_i^{e_i}x_{i+1}^{e_{i+1}} \cdots x_n^{e_n} \mbox{ and } w' =  x_1^{e_1}\cdots x_{i+1}^{e_{i+1}}x_i^{e_i} \cdots x_n^{e_n}.\]
For $j \neq i$ or $i+1$, we clearly have $X^w(x_j^{e_j}) = X^{w'}(x_j^{e_j})$, and so we must show
\[ X^w(x_i^{e_i}) = X^{w'}(x_i^{e_i}) \quad \mbox{ and } \quad X^w(x_{i+1}^{e_{i+1}}) = X^{w'}(x_{i+1}^{e_{i+1}}).\]
Interchanging the roles of $w$ and $w'$, it suffices to prove just one of these equations, say
$X^w(x_i^{e_i}) = X^{w'}(x_i^{e_i})$.

We have
\[X^w(x_i^{e_i}) = \phi_\mathbb F(x_1^{e_1} \cdots x_{i-1}^{e_{i-1}})(X_{J(i)})\]
whereas
\[X^{w'}(x_i^{e_i}) = \phi_\mathbb F(x_1^{e_1} \cdots x_{i-1}^{e_{i-1}}x_{i+1}^{e_{i+1}})(X_{J(i)}) = \phi_\mathbb F(x_1^{e_1} \cdots x_{i-1}^{e_{i-1}})\phi_\mathbb F(x_{i+1}^{e_{i+1}})(X_{J(i)}).\]
Since $x_i^{e_i}$ and $x_{i+1}^{e_{i+1}}$ commute, $X_{J(i+1)}$, the support of $\phi_\mathbb
F(x_{i+1}^{e_{i+1}}) = f_{J(i+1)}^{e_{i+1}}$ is disjoint from $X_{J(i)}$.  Therefore,
\[\phi_\mathbb F(x_{i+1}^{e_{i+1}})(X_{J(i)}) = X_{J(i)}\]
and the lemma follows.
\end{proof}

By this lemma we can unambiguously define $X^\sigma = X^w$, independent of the choice of $w \in
\Min(\sigma)$.

The main technical theorem we prove is the following.  From this, together with Theorems \ref{T:mod
distance}, \ref{T:WP distance} and \ref{T:Teich distance}, our Theorem \ref{T:main2} (and hence
also Theorem \ref{T:main}) follows easily.

\begin{theorem} \label{T:main technical}
Suppose $\Gamma$ and $\mathbb X$ are as above and $\mu \in \M(S)$. Then there exists a constant $K
\geq K_0$ with the following property.

Suppose that $\mathbb F = \{f_1,\ldots,f_n\}$ is fully supported on $\mathbb X$ and that
$\tau_{X_j}(f_j) \geq 2K$ for all $1 \leq j \leq n$, and let $\phi_{\mathbb F}:G \to \Mod(S)$ be
the associated homomorphism.  Then, for any $\sigma \in G$ with $x_1^{e_1}\cdots x_k^{e_k} \in
\Min(\sigma)$ we have
\begin{enumerate}
\item $d_{X^\sigma(x_i^{e_i})}(\mu,\phi(\sigma)\mu) \geq  K|e_i|$ for each $i=1,\ldots,k$.
Consequently,
\[ X^\sigma(\syl(\sigma)) \subset \Omega(K,\mu,\phi(\sigma)\mu).\]
\item $X^\sigma(\cdot):\syl(\sigma) \to \Omega(K,\mu,\phi(\sigma)\mu)$ is an order-preserving injection.
\end{enumerate}
\end{theorem}

\begin{proof}
Let
\[ K = K_0 + 20 + 2 \cdot \max\{d_{X_j}(\partial X_i,\mu): i \neq j\}.\]
Throughout the proof, we let $\phi = \phi_\mathbb F$.

In what follows, we prove statements 1 and 2 separately.  For both, the proof is by induction on
the number of syllables in $w \in \Min(\sigma)$.

\begin{proof}[Proof of Statement 1]
To make the ideas in the proof more transparent, we introduce simplified notation.  Given $w =
x_1^{e_1}\cdots x_k^{e_k} \in \Min(\sigma)$, define \[g_i = \phi(x_i^{e_i}) = f^{e_i}_{J(i)} \quad
\mbox{ and } \quad Y_i = X_{J(i)}.\]  Then $X^\sigma(x_1^{e_1}) = Y_1$, $X^\sigma(x_2^{e_2}) =
g_1Y_2$, and in general $X^\sigma(x_i^{e_i}) = g_1g_2\cdots g_{i-1}Y_i$.  In this notation,
statement 1 claims that \[d_{g_1\cdots g_{i-1}Y_i}(\mu,g_1\cdots g_k\mu) \geq K|e_i|,\] for $i = 2,
\dots, k$, and also $d_{Y_1}(\mu,g_1\cdots g_k\mu) \geq K|e_1|$.

Suppose $w$ has only one syllable.  Then the claim only states that $d_{Y_1}(\mu,g_1\mu) \geq
K|e_1|$, which holds because, letting $j = J(1)$, we know
\begin{eqnarray*}
d_{Y_1}(\mu,g_1\mu) & = & d_{X_j}(\mu,f^{e_1}_j(\mu))\\
 & = & \diam_{X_j}(\pi_{X_j}(\mu) \cup \pi_{X_j}((f_j^{e_1})\mu))\\
 & = & \diam_{X_j}(\pi_{X_j}(\mu) \cup f_j^{e_1}(\pi_{X_j}(\mu))) \\
 & \geq & \tau_{X_j}(f_j^{e_1}) \\
 & \geq & 2K|e_1|. \\ \end{eqnarray*}

Now suppose we have proved the claim for elements of $G=G(\Gamma)$ whose minimal representatives
have at most $k-1$ syllables.  Let $\sigma \in G$ with $w = x_1^{e_1} \cdots x_k^{e_k} \in
\Min(\sigma)$ having $k$ syllables.  Define $g_i$ and $Y_i$ as above.  Our next step is to separate
the product $g_1 \cdots g_k$ into subproducts, as illustrated below, with the additional
possibility that $a$, $b$, or $c$ might be the empty word:
\[ \overbrace{g_1 \cdots g_{\ell}}^{a}\overbrace{g_{\ell+1}\cdots g_{i-1}}^{b}g_i\overbrace{g_{i+1} \cdots g_k}^{c},\]
The subproducts $a$, $b$, and $c$ are defined as follows.  By Lemma \ref{L:Xsigma}, we may assume
that either $g_i$ and $g_{i+1}$ fail to commute, or by replacing $w$ with another word in
$\Min(\sigma)$, that $i = k$. In the first case, let $c = g_{i+1} \cdots g_k$; in the latter case,
let $c$ be the identity. If there exists some syllable to the left of $g_i$ which does not commute
with $g_i$, let $\ell$ be the largest index such that $g_i$ and $g_\ell$ do not commute, and let $a
= g_1 \cdots g_{\ell}$. Otherwise let $a$ be the identity. Let $b = g_{\ell+1} \cdots g_{i-1}$ if
$\ell +1 < i$, and otherwise let $b$ be the identity; observe that by construction, $b$ commutes
with $g_i$.

Because $g_1 \cdots g_k = abg_i c$, we have \[d_{g_1\cdots g_{i-1}Y_i}(\mu,g_1\cdots g_k\mu) =
d_{abY_i}(\mu,abg_{i}c\mu)=d_{Y_i}(b^{-1}a^{-1}\mu,g_{i}c\mu).\]

By the triangle inequality and the fact that $d_{Y_i}(g_{i}c\mu,c\mu) \geq 2K|e_i|$,
\[d_{Y_i}(b^{-1}a^{-1}\mu,g_{i}c\mu) \geq 2K|e_i| -
d_{Y_i}(b^{-1}a^{-1}\mu,c\mu).\]

To control the last term we again employ the triangle inequality:
\[d_{Y_i}(b^{-1}a^{-1}\mu,c\mu) \leq d_{Y_i}(b^{-1}a^{-1}\mu,\mu) + d_{Y_i}(\mu,c\mu).\]

Because $b$ is the (possibly empty) product of syllables $g_j$ that commute with $g_i$, $b$ acts as
the identity on $Y_i$.  Therefore we have
\begin{eqnarray*}
d_{Y_i}(b^{-1}a^{-1}\mu,\mu) & = & d_{Y_i}(a^{-1}\mu,b\mu) \\
 & = & \diam_{Y_i}(\pi_{Y_i}(a^{-1}\mu) \cup \pi_{Y_i}(b\mu)) \\
 & = & \diam_{Y_i}(\pi_{Y_i}(a^{-1}\mu) \cup \pi_{Y_i}(\mu)) \\
 & = & d_{Y_i}(a^{-1}\mu,\mu).\end{eqnarray*}

So far, we have shown
\[d_{g_1\cdots g_{i-1}Y_i}(\mu,g_1\cdots g_k\mu) \geq 2K|e_i| - d_{Y_i}(a^{-1}\mu,\mu) - d_{Y_i}(\mu,c\mu).\]

To finish, we prove that the last two terms on the right are each less than $K/2$.  Since the sign
of the $e_i$ never comes into play, the proof is very similar for either term, so we focus on
$d_{Y_i}(\mu,c\mu)$.  If $c$ is the identity, then $d_{Y_i}(\mu,c\mu)=\diam_{Y_i}(\mu) \leq 2 \leq
K/2$.  Otherwise $c = g_{i+1}\cdots g_k$.  Because subwords of minimal words are also minimal, $c =
\phi(\sigma_c)$ for some $\sigma_c \in G$ with minimal word $x_{i+1}^{e_{i+1}} \cdots x_k^{e_k}$,
which has strictly less than $k$ syllables.  Let $x'_1 = x_{i+1}^{e_{i+1}}$ be the first syllable.
Applying the induction hypothesis, we have
\begin{eqnarray*}
d_{Y_{i+1}}(\mu,g_{i+1}\cdots g_k\mu) & = & d_{Y_{i+1}}(\mu,c\mu)\\
 & = & d_{X^{\sigma_c}(x'_1)}(\mu,\phi(\sigma_c)\mu)\\
 & \geq & K|e_{i+1}|.\end{eqnarray*}
By our choice of $K$, $d_{Y_{i+1}}(\mu,\partial Y_i) = d_{X_{J(i+1)}}(\mu,\partial X_{J(i)}) \leq
K/2$.  Since $g_i$ and $g_{i+1}$ do not commute, $Y_i \pitchfork Y_{i+1}$, so we may apply the
triangle inequality to obtain
\begin{eqnarray*}
d_{Y_{i+1}}(\partial Y_i,g_{i+1}\cdots g_k\mu) & \geq & K|e_{i+1}| - d_{Y_{i+1}}(\mu,\partial Y_i) \\
& \geq & K|e_{i+1}| - K/2 \geq K/2 \geq 20/2 = 10.\end{eqnarray*} On the other hand, appealing to
Proposition \ref{P:Behrstock}, we know that $d_{Y_i}(\partial Y_{i+1},g_{i+1}\cdots g_k\mu) \leq
4$. By our choice of $K$, $d_{Y_i}(\mu,\partial Y_{i+1}) = d_{X_{J(i)}}(\mu,\partial X_{J(i+1)})
\leq K/2 - 4$, so combining, we have
\begin{eqnarray*}
d_{Y_i}(\mu,c\mu)
 & \leq & d_{Y_i}(\mu,\partial Y_{i+1}) + d_{Y_i}(\partial Y_{i+1},c\mu)\\
 & = &    d_{Y_i}(\mu,\partial Y_{i+1}) + d_{Y_i}(\partial Y_{i+1},g_{i+1}\cdots g_k\mu)\\
 & \leq & K/2 - 4 + 4 = K/2.\end{eqnarray*}
The entire argument can be mirrored for $d_{Y_i}(a^{-1}\mu,\mu)$, starting with the observation
that either $a^{-1}$ is the identity or $a^{-1} = g_{\ell}^{-1}\cdots g_{1}^{-1}$, where $g_{\ell}$
and $g_i$ do not commute.

To summarize, we have shown
\[d_{g_1\cdots g_{i-1}Y_i}(\mu,g_1\cdots g_k\mu) \geq 2K|e_i| - K/2 - K/2 \geq K|e_i|,\]
completing the induction for statement 1.
\end{proof}

\begin{proof}[Proof of Statement 2]
We will now show that $X^\sigma$ is an order-preserving injection. Notice that, by induction, we
need only show $X^\sigma(x_1^{e_1}) \neq X^\sigma(x_k^{e_k})$ to establish injectivity.

The subwords $w_{init}=x_1^{e_1} \cdots x_{k-1}^{e_{k-1}}$ and $w_{term}= x_2^{e_2} \cdots
x_k^{e_k}$ of $w$ are clearly minimal representatives of the elements $\sigma_{init},\sigma_{term}
\in G$ they represent.  Furthermore, the partial order on the syllables of $w_{init}$ and
$w_{term}$ is the restriction of the partial order on the syllables of $w$.

By the inductive hypothesis, the conclusion of the theorem holds for $\sigma_{init}$ and
$\sigma_{term}$. By construction we have
\[X^\sigma(x_i^{e_i}) = \left\{ \begin{array}{ll} X^{\sigma_{init}}(x_i^{e_i}) & \quad \mbox{ if } i \neq k\\
    \phi(x_1^{e_1})(X^{\sigma_{term}}(x_i^{e_i})) & \quad \mbox{ if }
    i \neq 1\\ \end{array} \right. .\] If $i$ is
neither $1$ nor $k$, then the two defining expressions are indeed equal.

Suppose $x_i^{e_i} \prec x_j^{e_j}$ for two syllables of $\sigma$. If $j \neq k$, then both are
syllables of $\sigma_{init}$ and hence by induction $X^{\sigma_{init}}(x_i^{e_i}) \prec
X^{\sigma_{init}}(x_j^{e_j})$.  Thus $X^{\sigma_{init}}(x_i^{e_i}) \pitchfork
X^{\sigma_{init}}(x_j^{e_j})$ and
$d_{X^{\sigma_{init}}(x_i^{e_i})}(\mu,X^{\sigma_{init}}(x_j^{e_j})) \geq 10$.  Since
$X^{\sigma_{init}}(x_i^{e_i}) = X^\sigma(x_i^{e_i})$ and $X^{\sigma_{init}}(x_j^{e_j}) =
X^\sigma(x_j^{e_j})$, we have $X^\sigma(x_i^{e_i}) \prec X^\sigma(x_j^{e_j})$.  A similar argument
works if $i \neq 1$ after applying $\phi(x_1^{e_1})$.

Now suppose $x_1^{e_1} \prec x_k^{e_k}$.  There are two cases.

\medskip \noindent \textbf{Case 1.} There is a syllable $x_i^{e_i}$
such that $x_1^{e_1} \prec x_i^{e_i} \prec x_k^{e_k}$.

\medskip \noindent Arguing as above, by induction we have
$X^\sigma(x_1^{e_1}) \prec X^\sigma(x_i^{e_i})$ and $X^\sigma(x_i^{e_i}) \prec X^\sigma(x_k^{e_k})$
in $\Omega(K,\mu,\phi(\mu))$ and hence by Corollary \ref{C:partialorder} we have that
$X^\sigma(x_1^{e_1}) \prec X^\sigma(x_k^{e_k})$.

\medskip \noindent \textbf{Case 2.}   There is no syllable $x_i^{e_i}$
such that $x_1^{e_1} \prec x_i^{e_i} \prec x_k^{e_k}$.

\medskip \noindent Then
there is a word $w \in \Min(\sigma)$ of the form:
\[ w = w_1 x_1^{e_1}x_k^{e_k} w_2 \] where $[w_1,x_1] =
1$ and $[x_k,w_2] = 1$ (and either or both of  $w_i$ may be the empty word).  Now:
\begin{eqnarray*}
d_{X_{J(1)}}(\mu, \partial X^\sigma(x_k^{e_k})) & = &
d_{X_{J(1)}}(\mu,\phi(w_1)\phi(x_1^{e_1})\partial X_{J(k)})\\
 & = & d_{X_{J(1)}}(\mu,\phi(x_1^{e_1})\partial X_{J(k)}) \\
 & \geq & d_{X_{J(1)}}(\mu,\phi(x_1^{e_1})\mu) - d_{X_{J(1)}}(\phi(x_1^{e_1})\mu,\phi(x_1^{e_1})\partial
 X_{J(k)})\\
 & \geq & 2K|e_1| - (K - 20) \geq 10 \end{eqnarray*}
 where the second equality comes from the fact that $w_1$ commutes with $x_1$, and so $\phi(w_1)$
 is the identity on $X_{J(1)}$.
Thus, by Proposition \ref{P:orderdescribed} (2), we have $X^\sigma(x_1^{e_1}) \prec
X^\sigma(x_k^{e_k})$.

\medskip In particular, if $x_1^{e_1} \prec x_k^{e_k}$ then $X^\sigma$
is injective.

All that remains now is to show that if $x_1^{e_1}$ and $x_k^{e_k}$ are not comparable by $\prec$,
then $X^\sigma(x_1^{e_1}) \neq X^\sigma(x_k^{e_k})$.  If they are not comparable, then $\sigma$ is
represented by a word of the form $w = w_1 x_1^{e_1} x_k^{e_k} w_2$, where as above $[w_1,x_1] = 1$
and $[x_k,w_2] = 1$.  Furthermore, $x_1^{e_1}$ and $x_k^{e_k}$ must commute, and hence it is clear
that $X^\sigma(x_k^{e_k})$ is disjoint from $X^\sigma(x_1^{e_1})$. In particular,
$X^\sigma(x_1^{e_1}) \neq X^\sigma(x_k^{e_k})$.
\end{proof}

This completes the proof of the Theorem.
\end{proof}

\noindent We can now prove the main theorem.
\begin{main2}
 Given a graph $\Gamma$ and a nice realization $\mathbb X =
\{X_1,\ldots,X_n\}$ of $\Gamma$ in $S$, there exists a constant $C > 0$ with the following
property. If $\mathbb F = \{f_1,\ldots,f_n\}$ is fully supported on $\mathbb X$ and
$\tau_{X_i}(f_i) \geq C$ for all $i = 1,\ldots,n$, then the associated homomorphism
\[ \phi_\mathbb F:G(\Gamma) \to \Mod(S) \]
is a quasi-isometric embedding.  Furthermore, the orbit map $G \to \T(S)$ is a quasi-isometric
embedding for both $d_\T$ and $d_\WP$.
\end{main2}

\begin{proof}
We first prove that given $\mu \in \M(S)$, we can choose $C > 0$ so that if $\tau_{X_i}(f_i) \geq C$ for each $i$, then
the orbit map $G(\Gamma) \to \M(S)$ given by $\sigma \mapsto \phi_\BF(\sigma) \mu$ is a quasi-isometric embedding. Since the orbit
map $\Mod(S) \to \M(S)$ is a quasi-isometry, this will suffice to prove the first statement.  Let
$K > 0$ be as in the proof of Theorem \ref{T:main technical} and let $C = 2K$.

First observe that for {\em any} metric space $(\frak X,d)$, any $x \in \frak X$ and any
$\sigma,\tau \in G(\Gamma)$, the triangle inequality implies
\[ d(\sigma \cdot x,\tau \cdot x) \leq A \, d_G(\sigma,\tau) \]
as long as $A \geq \max\{d(s_i \cdot x,x)\}_{i=1}^n$ (here $s_i$ are the generators for $G$).

In particular, given $\mu \in \M(S)$, to prove that the orbit map to $\M(S)$ is a quasi-isometry,
it suffices to find $A \geq 1$ and $B \geq 0$ so that
\[ d_G(1,\sigma) \leq A \, d_\M(\mu,\phi_\mathbb F(\sigma)(\mu)) + B\]
for all $\sigma \in G$ (then we further increase $A$ if necessary so that $A \geq \max\{d_\M(\mu,\phi_\BF(s_i)\mu) \}_{i=1}^n$).

Since $K \geq K_0$, from Theorem \ref{T:mod distance} there exists $A$ and $B$ so that for all
$\sigma \in G$,
\[ \sum_{X \in \Omega(K,\mu,\phi_\mathbb F(\sigma)(\mu))} d_X(\mu,\phi_\mathbb F(\sigma)(\mu)) \leq  A \, d_\M(\mu,\phi_\mathbb F(\sigma)(\mu)) + B.\]
On the other hand, if we let $w = x_1^{e_1} \ldots x_k^{e_k} \in \Min(\sigma)$ then by Theorem
\ref{T:main technical}, since $\tau_{X_i}(f_i) \geq C = 2K$ (and since $K \geq 1$) we have
\begin{eqnarray*} \label{E:the upper bound main}
d_G(1,\sigma)  & = & \displaystyle{ \sum_{i=1}^k |e_i| \quad \leq \quad \sum_{i=1}^k K |e_i|} \\
 & \leq &  \displaystyle{ \sum_{i=1}^k d_{X^\sigma(x_i^{e_i})}(\mu,\phi_\mathbb F(\sigma)(\mu))} \\
 & \leq  & \displaystyle{ \sum_{X \in \Omega(K,\mu,\phi_\mathbb F(\mu))} d_X(\mu,\phi_\mathbb
 F(\sigma)(\mu)).}\\
 & \leq & \displaystyle{A \, d_\M(\mu,\phi_\mathbb F(\sigma)(\mu)) + B}\\
 \end{eqnarray*}
which completes the proof of the first statement.

The proof of the statements regarding Teichm\"uller space are essentially identical.  For this, we
observe that the topological types of the surfaces in $X^\sigma(\syl(\sigma))$ are the same as
those of $\mathbb X$, and hence all are nonannular.  That is, $X^\sigma(\syl(\sigma)) \subset
\Omega_n(K,\mu,\phi_\mathbb F(\mu))$.  So, the above proof can be carried out replacing the use of
Theorem \ref{T:mod distance} with the use of Theorems \ref{T:WP distance} and \ref{T:Teich
distance}.
\end{proof}

\section{Elements of the Constructed Subgroups} \label{S:elements}

We now assume the hypothesis of Theorem \ref{T:main2} (and hence also Theorem \ref{T:main
technical}) on $\Gamma$, $\BX = \{ X_1,\ldots,X_n\}$, $C = 2K > 0$, and $\BF = \{f_1,\ldots,f_n\}$,
and let
\[ \phi_\BF:G(\Gamma) \to \Mod(S) \]
denote the associated homomorphism.  In this section we describe, in terms of the Thurston
classification, the mapping class image of any $\sigma \in G(\Gamma)$.  In particular, we identify
all pseudo-Anosov elements in the image.

Conjugate elements in $G(\Gamma)$ map to conjugate elements in $\Mod(S)$, and conjugation preserves
mapping class type, displacing the support of a mapping class by the homeomorphism corresponding to
the conjugating element.  Therefore to understand the image of $\sigma \in G(\Gamma)$, we may
assume it is an element with the minimal number of syllables among members of its conjugacy class.
We represent $\sigma$ by a word $w \in \Min(\sigma)$.  By changing the indices if necessary, we can
assume that $w$ is a word in the first $r = r(\sigma)$ generators $s_1,\ldots,s_r$, and $r$ is the
least number of generators needed to write $w$.\\

\noindent {\bf Remark.} In what follows, we will always assume that the indices on the
generators are of this type for the particular element $\sigma$ we are interested in.\\

We write $\Fill(X_1,\ldots,X_r)$ to denote the minimal union of subsurfaces, ordered by inclusion, which
contains all of the subsurfaces $X_1,\ldots,X_r$.  Alternatively, $\Fill(X_1,\ldots,X_r)$ is the unique union of subsurfaces containing $X_1 \cup \ldots \cup X_r$ with the property that for every essential curve $\gamma$ contained in it, the projection to at least one of
$X_1,\ldots,X_r$ is nontrivial.  Write $\Fill_\BX(\sigma) = \Fill(X_1,\ldots,X_r)$.

Now, if $\sigma' = \sigma_0 \sigma \sigma_0^{-1}$, then we define $\Fill_\BF(\sigma') =
\phi_\BF(\sigma_0)(\Fill_\BX(\sigma))$.  Note that if $\sigma_0$ is the identity so $\sigma' =
\sigma$, then $\Fill_\BF(\sigma) = \Fill_\BX(\sigma)$ depends only on $\BX$, whereas otherwise, it
depends on $\BF$.

It follows easily from the discussion above that for any $\sigma$, $\phi_\BF(\sigma)$ is supported
on $\Fill_\BF(\sigma)$.  That is, $\phi_\BF(\sigma)$ is represented by a homeomorphism
which is the identity outside $\Fill_\BF(\sigma)$.   In this section, we prove the following.

\begin{theorem}\label{T:find pAs}
Suppose $\Gamma$, $\mathbb X = \{X_1,\ldots,X_n\}$, $C = 2K > 0$, $\mathbb F = \{f_1,\ldots,f_n\}$ satisfy the hypotheses of Theorem \ref{T:main2}, and let
\[ \phi_\BF = \phi:G(\Gamma) \to \Mod(S)\]
denote the associated homomorphism.
Then $\phi(\sigma)$ is pseudo-Anosov on each component of $\Fill_\BF(\sigma)$. In particular,
$\phi(\sigma)$ is pseudo-Anosov if and only if $\Fill_\BF(\sigma) = S$.
\end{theorem}

Before we begin the proof, we explain a few reductions which will greatly simplify the
exposition.  First, as remarked above, we need only consider the case that $\sigma$ has the minimal
number of syllables among all its conjugates, so we assume this is the case from now on. Therefore
$\Fill_\BF(\sigma) = \Fill_\BX(\sigma)$.

Next, we wish to further reduce to the case that $\Fill_\BX(\sigma)$ is connected.  To describe
this reduction, first let $\Gamma'$ denote the subgraph spanned by $s_1,\ldots,s_r$.  Since the
other generators play no role in this discussion, we assume as we may, that $\Gamma' = \Gamma$. Let
$\Gamma^c$ be the {\em complement} of $\Gamma$.   That is, $\Gamma^c$ is the graph with the same
vertex set as $\Gamma$ and where two vertices span an edge in $\Gamma^c$ if and only if they do not
span an edge in $\Gamma$. Note that generators/vertices $s_i$ and $s_j$ in different components of
$\Gamma^c$ commute. Therefore, we may write $\sigma = \sigma_1 \cdots \sigma_\ell$, where $\ell$ is
the number of components of $\Gamma^c$ and each $\sigma_i$ is in the group generated by vertices in
a single component of $\Gamma^c$.  In particular, $[\sigma_i,\sigma_j] = 1$ for all $i$ and $j$.

Now observe that the vertices of a path in $\Gamma^c$ corresponds to a chain of overlapping
subsurfaces in $S$, and hence the components of $\Fill_\BX(\sigma)$ correspond precisely to the
components of $\Gamma^c$.  In fact, one easily checks that each $\sigma_i$ also has the least
number of syllables in its conjugacy class, and
$\{\Fill_\BX(\sigma_1),\ldots,\Fill_\BX(\sigma_\ell)\}$ is precisely the set of components of
$\Fill_\BX(\sigma)$.  So, restricting attention to one of the subwords $\sigma_i$, we may assume
that $\Fill_\BX(\sigma)$ is connected.

Finally, we note that we can in fact restrict to the case that $\Fill_\BX(\sigma) = S$.  To see
that this is possible, note that $\phi(G(\Gamma))$ is the identity outside $S' =
\Fill_\BX(\sigma)$. So, $\phi$ ``restricts'' to a homomorphism
\[\hat \phi:G(\Gamma) \to \Mod(S')\]
and $\hat \phi(\sigma)$ is pseudo-Anosov if and only if $\phi(\sigma)$ is pseudo-Anosov on
$S' = \Fill_\BX(\sigma)$.\\

We now set out to prove Theorem \ref{T:find pAs} assuming (1) that $\sigma$ has the least number of
syllables in its conjugacy class and (2) $\Fill_\BX(\sigma) = S$.  The proof makes use of the
partial order on $\syl(\sigma)$ and $\syl(\sigma^n)$, and the order-preserving injection
$X^\sigma(\cdot)$ of the previous section. Regarding these, let us set down a series of lemmas.

\begin{lemma}\label{lm:fill}
For $\sigma$ as above, $\Fill(X^{\sigma}(\syl(\sigma)))=S$.
\end{lemma}

\begin{proof}
  Fix a minimal word $x_1^{e_1}\cdots x_k^{e_k} \in \Min(\sigma)$.
  Given a curve $\gamma \subset S$, we must show that $\gamma$
  intersects some $X^{\sigma}(x_i^{e_i}) = \phi(x_1^{e_1}\cdots
  x_{i-1}^{e_{i-1}})X_{J(i)}$.  As $\Fill_\BX(\sigma) = \Fill(X_1,\dots,X_r)=S$,
  the curve $\gamma$ intersects some $X_j$.  Let $i$ be the minimal
  index such that $\gamma$ intersects $X_{J(i)}$.

  If $i = 1$, then $\gamma$ intersects $X_{J(1)} =
  X^{\sigma}(x_1^{e_1})$ and the lemma holds.  Else, notice that
  $\phi(x_1^{e_1}\cdots x_{i-1}^{e_{i-1}})(\gamma) = \gamma$.  Hence
  $\gamma$ intersects $X^{\sigma}(x_i^{e_i})$.
\end{proof}

\begin{lemma}\label{lm:minimalconjugate}
  For $\sigma$ as above, $w \in \Min(\sigma)$, and $n \in \mathbb{Z}$ we have $w^n \in
  \Min(\sigma^n)$.
\end{lemma}

\begin{proof}
  Clearly $w^n$ represents the element $\sigma^n$; what needs to be
  shown is that this word is minimal.

  Write $w = x_1^{e_1} \cdots x_k^{e_k}$ and assume that $w^n$ is not
  a minimal word representing $\sigma^n$.  Thus we have a sequence of
  the three moves described in Section \ref{S:normal forms} which
  reduces the number of syllables in $w^k$.  We can label the
  syllables of $w^n$ by $x_{1,1}^{e_1}\cdots
  x_{1,k}^{e_k}x_{2,1}^{e_1}\cdots x_{n,k}^{e_k}$ where each block
  $x_{j,1}^{e_1}\cdots x_{j,k}^{e_k} = w$.  As $w$ is minimal, each of
  the $e_i \neq 0$, hence in order to reduce the number of syllables
  of $w^k$ we have some sequence of type (3) moves followed by a type
  (2) move.

  Also as $w$ is minimal, the type (2) cannot occur between syllables
  of the form $x_{j,i}^{e_i}$ and $x_{j,i'}^{e_{i'}}$.  Therefore,
  after applying some type (3) moves, we have a type (2) move between
  syllabes of the form $x_{j,i}^{e_i}$ and $x_{j',i'}^{e_{i'}}$ where
  $j < j'$.  We claim that we can assume that $j' = j+1$.  For
  if not, then $[x_i,x_\ell] = 1$ for all $\ell$ and hence after a
  sequence of type (3) moves we could apply the move
  $x_i^{e_i}x_{i'}^{e_{i'}} \mapsto x_i^{e_i + e_{i'}}$, contradicting
  the fact that $w$ is minimal.

  As the set of indices $\ell$ such that $[x_{j,i},x_\ell] = 1$ and
  the set of indices $\ell$ such that $[x_{j',i'},x_\ell] = 1$ are the
  same, the above assumptions give a sequence of type (3) moves on $w$
  such that brings $w$ to a word of the form $x_{i'}^{e_{i'}}x_1^{e_1}
  \cdots x_k^{e_k}x_i^{e_i}$.  But now conjugating $\sigma$ by
  $x_i^{e_i}$ results in an element with fewer syllables than $\sigma$
  which contradicts our assumption that $|\syl(\sigma)|$ is minimal
  among conjugates of $\sigma$. Thus $w^n \in \Min(\sigma^n)$.
\end{proof}

The above lemma allows us to define syllable shift maps $\sigma_n \co \syl(\sigma^n) \to
\syl(\sigma^{n+1})$ by $\sigma_n(x_{j,i}^{e_i}) = x_{j+1,i}^{e_i}$ using the notation from the
proof of the lemma. Notice, the maps $\sigma_n$ preserve the partial order. For $n > m$ we use the
notation $\sigma_{m,n} = \sigma_{n-1} \cdots \sigma_m$. The map $\sigma_{m,n} \co \syl(\sigma^m)
\to \syl(\sigma^n)$ shifts syllables by $n-m$ blocks and also preserves the partial order.

Under the assumption that $\Fill(X_1,\dots,X_r)=S$, we have that $\Gamma^c$ is connected. In
particular for any syllable $x_i^{e_i}$, there is another syllable $x_j^{e_j}$ such that $[x_i,x_j]
\neq 1$.

\begin{lemma}\label{lm:basecase}
  For $\sigma$ as above, and all $x_i^{e_i} \in \syl(\sigma)$, we have $x_i^{e_i} \prec
  \sigma_{1,2}(x_i^{e_i}) \in \syl(\sigma^2)$.
\end{lemma}

\begin{proof}
  This is similar to the proof of Lemma \ref{lm:minimalconjugate}.  If
  the conclusion is wrong, then $[x_i,x_j] = 1$ for all syllables
  $x_j^{e_j} \in \syl(\sigma)$.  This contradicts the fact that
  $\Gamma^c$ is connected.
\end{proof}

\begin{lemma}\label{lm:compare}
  For $\sigma$ as above, and all $x_i^{e_i},x_j^{e_j} \in \syl(\sigma)$, there exists $n$, $1 \leq n
  \leq r+1$, such that $x_i^{e_i} \prec \sigma_{1,n}(x_j^{e_j})$. In
  particular, for all syllables $x_i^{e_i}, x_j^{e_j} \in \syl(\sigma)$
  we have $x_i^{e_i} \prec \sigma_{1,r+1}(x_j^{e_j})$. \end{lemma}

\begin{proof}
  Fix a minimal sequence of generators $x_i = x_{i_1}, \ldots,
  x_{i_m} = x_j$ such that $ [x_{i_\ell} , x_{i_{\ell + 1}} ] \neq 1$.
  Such a sequence exists as $\Gamma^c$ is connected. Further,
  notice that $m \leq r$. We will prove the lemma by induction on $m$.
  Specifically, we will prove that if there is a path of length $m$
  between $v_{J(i)}$ and $v_{J(j)}$ in $\Gamma^c$, then $x_i^{e_i}
  \prec \sigma_{1,m+1}(x_j^{e_j})$.

  Suppose $m = 1$, hence as generators $x_i = x_j$. The case when
  $x_i^{e_i} = x_j^{e_j}$ as \emph{syllables} in $\syl(\sigma)$ is
  covered by Lemma
  \ref{lm:basecase}.  Else, we must have that $x_i^{e_i} \prec
  x_j^{e_j}$ or $x_j^{e_j} \prec x_i^{e_i}$.  In the first case using
  Lemma \ref{lm:basecase} we have $x_i^{e_i} \prec x_j^{e_j} \prec
  \sigma_{1,2}(x_j^{e_j})$.  In the second case, if $x_i^{e_i} \not \prec
  \sigma_{1,2}(x_j^{e_j})$ we can argue as in the proof of Lemma
  \ref{lm:minimalconjugate} that $|\syl(\sigma)|$ is not minimal among
  conjugates of $\sigma$.

  Now by induction, we have that $x_i^{e_i} \prec
  \sigma_{1,m}(x_{i_{m-1}}^{e_{i_{m-1}}})$. Since $[x_{i_{m-1}},x_j] \neq
  1$, we must have $\sigma_{1,m}(x_{i_{m-1}}^{e_{i_{m-1}}}) \prec
  \sigma_{1,m+1}(x_j^{e_j})$.  Hence $x_i^{e_i} \prec
  \sigma_{1,m}(x_{i_{m-1}}^{e_{i_{m-1}}}) \prec
  \sigma_{1,m+1}(x_j^{e_j})$. This completes the proof of the lemma.
\end{proof}

\begin{proof}[Proof of Theorem \ref{T:find pAs}]  We assume $\sigma$ is as above, so $\Fill_\BX(\sigma) = S$, and prove
$\phi(\sigma)$ is pseudo-Anosov.  For this, it suffices to prove the following.\\

\begin{claim} For every integer $N > 0$ we have
\[ d_S(\partial X_{J(1)},\phi(\sigma^{N(2r+1)})(\partial X_{J(1)})) \geq N.\]
\end{claim}
Indeed, this claim says that $\phi(\sigma)$ acts with positive translation length on $\C(S)$ as required.
\begin{proof}[Proof of claim]
According to Lemma \ref{lm:compare} we have
\[ x_1^{e_1} \prec \sigma_{1,r+1}(x_j^{e_j}) \prec \sigma_{r+1,2r+1}(\sigma_{1,r+1}(x_1^{e_1})) =
\sigma_{1,2r+1}(x_1^{e_1})
\]
Now, from the definitions, we see that $X^{\sigma^n} \circ \sigma_{1,n} = \phi(\sigma^n) \circ
X^{\sigma}$ for all $n > 1$, and since $X^\sigma$ and $X^{\sigma^n}$ are order preserving by
Theorem \ref{T:main technical} we also have
\[
X_{J(1)} \prec \phi(\sigma^{r+1})(X^{\sigma}(x_j^{e_j})) \prec \phi(\sigma^{2r+1})(X_{J(1)})
\]

This implies that no curve $\gamma \subset S$ is disjoint from both $\partial X_{J(1)}$ and
$\phi(\sigma^{2r+1})\partial X_{J(1)}$.  Indeed, suppose otherwise.  According to Lemma \ref{lm:fill}, the collection of subsurfaces $\phi(\sigma^{r+1})X^{\sigma}(\syl(\sigma))$ fill $S$, so there is some subsurface, say
$\phi(\sigma^{r+1})X^{\sigma}(x_j^{e_j})$ where $\gamma$ has nonempty projection. Hence,
\[
d_{\phi(\sigma^{r+1})X^{\sigma}(x_j^{e_j})}(\partial X_{J(1)}, \phi(\sigma^{2r+1})\partial X_{J(1)}) \leq 4.
\]
However, since $X_{J(1)} \prec \phi(\sigma^{r+1})(X^{\sigma}(x_j^{e_j})) \prec \phi(\sigma^{2r+1})(X_{J(1)})$ it follows from Proposition \ref{P:orderdescribed} and the triangle inequality that
\[
d_{\phi(\sigma^{r+1})X^{\sigma}(x_j^{e_j})}(\partial X_{J(1)},\phi(\sigma^{2r+1})\partial X_{J(1)}) \geq K -
8 > 4
\]
which is a contradiction.

By the same reasoning, no curve $\gamma$ can be disjoint from more than one of the surfaces
\[ \{ \partial X_{J(1)},\phi(\sigma^{2r+1})(\partial X_{J(1)}),\ldots,\phi(\sigma^{N(2r+1)})(\partial X_{J(1)}) \}.
\]
On the other hand, since $X_{J(1)} \prec \phi(\sigma^{\ell(r+1)})(X_{J(1)})) \prec \phi(\sigma^{N(2r+1)})(X_{J(1)})$ for all $0 < \ell < N$, Proposition \ref{P:orderdescribed} and the triangle inequality again imply
\[
d_{\phi(\sigma^{\ell(r+1)})(X_{J(1)})}(\partial
X_{J(1)},\phi(\sigma^{2r+1})\partial X_{J(1)}) \geq K - 8 \geq K_0
\]
where the last inequality comes from the choice of $K$ in the proof of Theorem \ref{T:main technical}.

Now, according to Theorem \ref{T:BGI} any geodesic in $\C(S)$ from $\partial X_{J(1)}$ to $\phi(\sigma^{N(2r+1)})(\partial X_{J(1)})$ must contain a curve disjoint from each $\phi(\sigma^{\ell(2r+1)})(\partial X_{J(1)})$, for each $j = 0,\ldots,N$.  Since these curves must all be distinct by the previous paragraph, we see that this geodesic contains at least $N+1$ vertices, so
\[ d_S(\partial X_{J(1)},\phi(\sigma^{N(2r+1)})(\partial X_{J(1)})) \geq N \]
as required.

This completes the proof of the claim, and also the proof of the Theorem.
\end{proof}
\end{proof}

\section{Surface subgroups} \label{S:surface_subgroups}

In this final section we prove the following corollary of Theorem \ref{T:main2} and briefly discuss surface subgroups of right-angled Artin subgroups of the mapping class group.

\begin{surfacesubgroup}
For any closed surface $S$ of genus at least $3$ and any $h \geq 2$, there exist infinitely many
nonconjugate genus $h$ surface subgroups of $\Mod(S)$, each of which act cocompactly on some
quasi-isometrically embedded hyperbolic plane in the Teichm\"uller space $\T(S)$, with either of
the standard metrics.
\end{surfacesubgroup}

\begin{proof}
Let $\Gamma$ be the cyclic graph of length $5$ and $G(\Gamma)$ the associated right-angled Artin
group.  It was shown in \cite{crispwiest1} that $G(\Gamma)$ contains a quasi-isometrically embedded
genus $2$ surface subgroup, and hence surface subgroups of all genus $h \geq 2$ (it had been
previously shown to contain a genus $5$ surface subgroup in \cite{servatiusdroms}).  As described in \cite{crispwiest1}, this example has a nice description as follows.

Suppose the generators of $G(\Gamma)$ are $a,b,c,d,e$ with $[e,a] = [a,b] = [b,c] = [c,d] = [d,e] = 1$.  Then the homomorphism from the fundamental group of a genus two surface to $G(\Gamma)$ is described by Figure \ref{F:curve_system} as follows.  The figure shows a system of curves on the surface with labels from the set $\{a,b,c,d,e\}$ and transverse orientations.  Choosing a basepoint in the complement of the curve system shown, a loop will cross the curves in the system, and one reads off an element of $G(\Gamma)$ according to the curves one crosses, and in which direction (crossing in the direction opposite the given transverse orientation, one should read an inverse of the generator); see \cite{crispwiest1} for more details.

\begin{figure}[htb]
\begin{center}
\includegraphics[height=1.5truein]{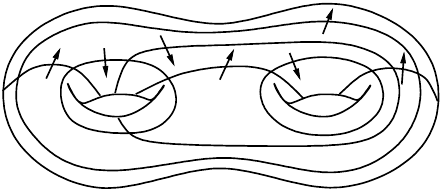}
\setlength{\unitlength}{0.01\linewidth}
\begin{picture}(0,0)
\put(-58,13){$c$}
\put(-33,13){$c$}
\put(-7,13){$e$}
\put(-50,20){$b$}
\put(-20,21.5){$d$}
\put(-42.5,20.5){$a$}
\put(-25,18){$d$}
\end{picture}
\caption{A curve system on a genus $2$ surface which defines an embedding into $G(\Gamma)$, where $\Gamma$ is the cyclic graph of length $5$.} \label{F:curve_system}
\end{center}
\end{figure}
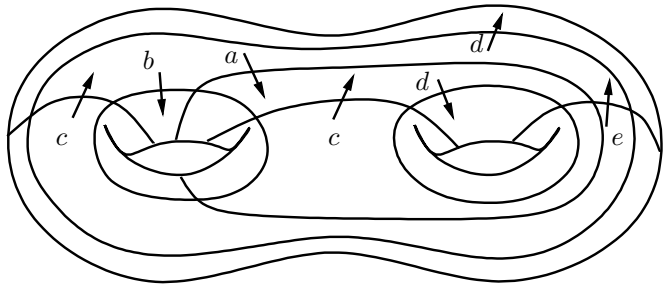

In Section \ref{S:realizing a graph} we observed that $\Gamma$ has a nice realization $\mathbb X =
\{X_1,\ldots,X_5\}$ in any closed surface $S$ of genus $g \geq 3$.  Let $C > 0$ be the constant
from Theorem \ref{T:main2} and $\mathbb F = \{f_1,\ldots,f_5\}$ be any mapping classes fully
supported on $\mathbb X$ with $\tau_{X_i}(f_i) \geq C$.  For every $n > 0$ let $\mathbb F^n =
\{f_1^n,\ldots,f_5^n\}$ so that we also have $\tau_{X_i}(f_i^n) \geq C$.  The family of
right-angled Artin subgroups $\phi_{\mathbb F^n}(G(\Gamma))$ necessarily contains infinitely many
distinct conjugacy classes---observe that the proof of Theorem \ref{T:main2} implies that the
minimal translation length on $\T(S)$ of any element of $\phi_{\mathbb F^n}(G(\Gamma))$ is tending
to infinity as $n \to \infty$.  Similarly, the set of surface subgroups described above, thought of
as subgroups of $\Mod(S)$ via the homomorphisms $\phi_{\mathbb F^n}$, have minimal translation
length on $\T(S)$ tending to infinity as $n \to \infty$.  Consequently, there are infinitely many
pairwise nonconjugate genus $h$ surface subgroups.

That each of these stabilizes a quasi-isometrically embedded hyperbolic plane $\mathbb H \subset
\T(S)$ follows from the fact that the surface group itself is quasi-isometric to $\mathbb H$, and
the orbit map defines a quasi-isometric embedding by Theorem \ref{T:main2}.  The surface group
clearly acts cocompactly on this plane.
\end{proof}

It follows that these surface groups all have positive translation length on Teichm\"uller space.
However, as we have already mentioned, they cannot be purely pseudo-Anosov.  In fact, for surface
subgroups of right-angled Artin groups, this is always the case.

\begin{proposition} \label{P:not purely pa}
Suppose $G(\Gamma) < \Mod(S)$ is a right-angled Artin subgroup and $\pi_1(\Sigma) < G(\Gamma)$ is a
surface subgroup.  Then as a subgroup of $\Mod(S)$, $\pi_1(\Sigma)$ contains a nontrivial reducible
element.
\end{proposition}
\begin{proof}
As was shown in \cite{crispwiest1}, every homomorphism from a surface group $\pi_1(\Sigma)$ into a
right-angled Artin group $G(\Gamma)$ arises as in the proof of the previous corollary.  That is,
there is a curve system on $\Sigma$, each curve is endowed with a transverse orientation, the
components are labeled by generators of $G(\Gamma)$, and the homomorphism is obtained by taking a
loop and reading off the generators as one crosses the curves in the system.

Furthermore, one may assume that each of the curves in the system is essential, and if the
homomorphism $H \to G(\Gamma)$ is injective then these curves cut $\Sigma$ into disks.  Now,
consider a loop $\gamma$ which runs parallel to, without crossing, one of the curves in the system.
Call this curve $\eta_1$ and suppose the associated generator of $G(\Gamma)$ is $s_1$.  The loop
$\gamma$ crosses other curves $\eta_{i_1},\ldots,\eta_{i_k}$ and so determines some word
$s_{i_1}^{\pm 1} \cdots s_{i_k}^{\pm 1}$, which is the image of $\gamma$ in $G(\Gamma)$.   Choosing
$\gamma$ to run very close to $\eta_1$, we can assume that the curves $\eta_{i_1},\ldots,\eta_{i_k}$
which $\gamma$ crosses are also nontrivially intersected by $\eta_1$.  As noted in
\cite{crispwiest1}, each of the associated generators $s_{i_1},\ldots,s_{i_k}$ must commute
with $s_1$, and be different from it.

Now we are essentially done.  The image of $\gamma$ is an element which commutes with $s_1$, and in
fact, the image of $\gamma$ and $s_1$ in $G(\Gamma)$ generate a subgroup isomorphic to $\mathbb
Z^2$.  If $G(\Gamma)< \Mod(S)$, then the image of $\gamma$ in $\Mod(S)$ has centralizer which
contains $\mathbb Z^2$. As is well known, the image of $\gamma$ cannot be pseudo-Anosov; see
\cite{ivanov}.
\end{proof}

\noindent {\bf Remark.} In fact, the assumption that $\pi_1(\Sigma)$ is a surface group can be
relaxed to the assumption that $\pi_1(\Sigma)$ is a finitely presented $1$--ended group.\\

In \cite{koberda}, Koberda observes that $\Mod(S_g)$ is not commensurable with a right-angled Artin
group if $g \geq 3$ (in fact, he proves the stronger statement that $\Mod(S_g)$ cannot virtually
embed in a right-angled Artin group).  This is also true for genus $2$ as the following shows.

\begin{proposition}
The group $\Mod(S_2)$ is not commensurable with a right-angled Artin group.
\end{proposition}
\begin{proof}
Suppose $\Mod(S_2)$ is commensurable with $G(\Gamma)$, with $\Lambda$ isomorphic to a finite index
subgroup of both.  Let $\pi_1(\Sigma)  < \Mod(S_2)$ be a surface subgroup as constructed in
\cite{leiningerreid}.  In this surface group, there is exactly one element of $\pi_1(\Sigma)$, up
to conjugacy and powers, which is not pseudo-Anosov.  Moreover, this one element represents a
simple closed curve $\alpha$ on $\Sigma$.

Now, $\pi_1(\Sigma) \cap \Lambda$ is a finite index subgroup of $\pi_1(\Sigma)$ and so corresponds
to a covering space $p:\widetilde \Sigma \to \Sigma$, and we write $\pi_1(\widetilde \Sigma) <
\pi_1(\Sigma)$ for the image under $p_*$.  Note that the reducible elements of $\pi_1(\widetilde
\Sigma)$ in $\Mod(S_2)$ represent a finite set of pairwise disjoint simple closed curves on
$\widetilde \Sigma$, namely $p^{-1}(\alpha)$.

On the other hand, a closer inspection of the proof of  the previous proposition shows that, viewing
$\Lambda < G(\Gamma)$, there are actually {\em two} elements $\gamma_1,\gamma_2 \in \pi_1(\widetilde \Sigma)$
which represent curves on $\widetilde \Sigma$ which nontrivially intersect, whose centralizers in $G(\Gamma)$
contain $\mathbb Z^2$.  These must represent reducible elements in $\Mod(S)$, and this is a
contradiction.
\end{proof}

\noindent {\bf Remark.}  This same proof also works to show that the mapping class group of an $n$--punctured sphere, with $n \geq 6$ is not commensurable with any right-angled Artin group.  The point is that the examples from \cite{leiningerreid} can be chosen to descend to the quotient by the hyper-elliptic involution, and then one of the punctures can be erased (with the exception of the genus $2$ case).

\bibliography{raagqi}{}

\begin{thebibliography}{10}

\bibitem{atiyah}
M.~F. Atiyah.
\newblock The signature of fibre-bundles.
\newblock In {\em Global {A}nalysis ({P}apers in {H}onor of {K}. {K}odaira)},
  pages 73--84. Univ. Tokyo Press, Tokyo, 1969.

\bibitem{behr}
Jason Behrstock.
\newblock Asymptotic geometry of the mapping class group and teichmueller
  space.
\newblock {\em Geom. Topol.}, 10:1523--1578, 2006.

\bibitem{bkmm}
Jason Behrstock, Bruce Kleiner, Yair Minsky, and Lee Mosher.
\newblock Geometry and rigidity of mapping class groups.
\newblock Preprint, \href{http://arxiv.org/abs/0801.2006}{arXiv:0801.2006},
  2008.

\bibitem{bm}
Jason Behrstock and Yair Minsky.
\newblock Centroids and the rapid decay property in mapping class groups.
\newblock Preprint, \href{http://arxiv.org/abs/0810.1969}{arXiv:0810.1969},
  2008.

\bibitem{bowditch-atoroidal}
Brian~H. Bowditch.
\newblock Atoroidal surface bundles over surfaces.
\newblock {\em Geom. Funct. Anal.}, 19(4):943--988, 2009.

\bibitem{brock}
Jeffrey~F. Brock.
\newblock The {W}eil-{P}etersson metric and volumes of 3-dimensional hyperbolic
  convex cores.
\newblock {\em J. Amer. Math. Soc.}, 16(3):495--535 (electronic), 2003.

\bibitem{charney}
Ruth Charney.
\newblock An introduction to right-angled {A}rtin groups.
\newblock {\em Geom. Dedicata}, 125:141--158, 2007.

\bibitem{crispfarb}
John Crisp and Benson Farb.
\newblock The prevalence of surface subgroups in mapping class groups.
\newblock In preparation.

\bibitem{crispparis}
John Crisp and Luis Paris.
\newblock The solution to a conjecture of {T}its on the subgroup generated by
  the squares of the generators of an {A}rtin group.
\newblock {\em Invent. Math.}, 145(1):19--36, 2001.

\bibitem{crispsageevsapir}
John Crisp, Michah Sageev, and Mark Sapir.
\newblock Surface subgroups of right-angled {A}rtin groups.
\newblock {\em Internat. J. Algebra Comput.}, 18(3):443--491, 2008.

\bibitem{crispwiest1}
John Crisp and Bert Wiest.
\newblock Embeddings of graph braid and surface groups in right-angled {A}rtin
  groups and braid groups.
\newblock {\em Algebr. Geom. Topol.}, 4:439--472, 2004.

\bibitem{crispwiest2}
John Crisp and Bert Wiest.
\newblock Quasi-isometrically embedded subgroups of braid and diffeomorphism
  groups.
\newblock {\em Trans. Amer. Math. Soc.}, 359(11):5485--5503, 2007.

\bibitem{farbmosher}
Benson Farb and Lee Mosher.
\newblock Convex cocompact subgroups of mapping class groups.
\newblock {\em Geom. Topol.}, 6:91--152 (electronic), 2002.

\bibitem{fujiwara}
Koji Fujiwara.
\newblock Subgroups generated by two pseudo-{A}nosov elements in a mapping
  class group. {I}. {U}niform exponential growth.
\newblock In {\em Groups of diffeomorphisms}, volume~52 of {\em Adv. Stud. Pure
  Math.}, pages 283--296. Math. Soc. Japan, Tokyo, 2008.

\bibitem{gonzalezharvey}
G.~Gonz{\'a}lez-D{\'{\i}}ez and W.~J. Harvey.
\newblock Surface groups inside mapping class groups.
\newblock {\em Topology}, 38(1):57--69, 1999.

\bibitem{gordonlongreid}
C.~McA. Gordon, D.~D. Long, and A.~W. Reid.
\newblock Surface subgroups of {C}oxeter and {A}rtin groups.
\newblock {\em J. Pure Appl. Algebra}, 189(1-3):135--148, 2004.

\bibitem{green}
E.~Green.
\newblock Graph {P}roducts of {G}roups, {T}hesis.
\newblock The University of Leeds (1990).

\bibitem{hamenstadt}
Ursula Hamenst{\"a}dt.
\newblock {Word hyperbolic extensions of surface groups}.
\newblock Preprint,
  \href{http://arxiv.org/abs/math/0505244}{arXiv:math.GT/0505244}, 2005.

\bibitem{hermiller-meier}
Susan Hermiller and John Meier.
\newblock Algorithms and geometry for graph products of groups.
\newblock {\em J. Algebra}, 171(1):230--257, 1995.

\bibitem{hsu-wise}
Tim Hsu and Daniel~T. Wise.
\newblock On linear and residual properties of graph products.
\newblock {\em Michigan Math. J.}, 46(2):251--259, 1999.

\bibitem{ivanov}
Nikolai~V. Ivanov.
\newblock {\em Subgroups of {T}eichm\"uller modular groups}, volume 115 of {\em
  Translations of Mathematical Monographs}.
\newblock American Mathematical Society, Providence, RI, 1992.
\newblock Translated from the Russian by E. J. F. Primrose and revised by the
  author.

\bibitem{kentleininger}
Richard~P. Kent, IV and Christopher~J. Leininger.
\newblock Shadows of mapping class groups: capturing convex cocompactness.
\newblock {\em Geom. Funct. Anal.}, 18(4):1270--1325, 2008.

\bibitem{sanghyunkim1}
Sang-hyun Kim.
\newblock Co-contractions of graphs and right-angled {A}rtin groups.
\newblock {\em Algebr. Geom. Topol.}, 8(2):849--868, 2008.

\bibitem{sanghyunkim2}
Sang-hyun Kim.
\newblock On right-angled {A}rtin groups without surface subgroups.
\newblock {\em Groups Geom. Dyn.}, 4(2):275--307, 2010.

\bibitem{koberda}
Thomas Koberda.
\newblock {Right-angled Artin groups and a generalized isomorphism problem for
  finitely generated subgroups of mapping class groups}.
\newblock Preprint.

\bibitem{kodaira}
K.~Kodaira.
\newblock A certain type of irregular algebraic surfaces.
\newblock {\em J. Analyse Math.}, 19:207--215, 1967.

\bibitem{leiningerreid}
C.~J. Leininger and A.~W. Reid.
\newblock A combination theorem for {V}eech subgroups of the mapping class
  group.
\newblock {\em Geom. Funct. Anal.}, 16(2):403--436, 2006.

\bibitem{linch}
Michele Linch.
\newblock A comparison of metrics on {T}eichm\"uller space.
\newblock {\em Proc. Amer. Math. Soc.}, 43:349--352, 1974.

\bibitem{liusunyau1}
Kefeng Liu, Xiaofeng Sun, and Shing-Tung Yau.
\newblock Canonical metrics on the moduli space of {R}iemann surfaces. {I}.
\newblock {\em J. Differential Geom.}, 68(3):571--637, 2004.

\bibitem{liusunyau2}
Kefeng Liu, Xiaofeng Sun, and Shing-Tung Yau.
\newblock Canonical metrics on the moduli space of {R}iemann surfaces. {II}.
\newblock {\em J. Differential Geom.}, 69(1):163--216, 2005.

\bibitem{longreidthistlethwaite}
Darren Long, Alan Reid, and Morwen Thistlethwaite.
\newblock Zariski dense surface subgroups in {S}{L}(3,{{\bf Z}}).
\newblock Preprint.

\bibitem{mangahas}
Johanna Mangahas.
\newblock Uniform uniform exponential growth of subgroups of the mapping class
  group.
\newblock {\em Geom. Funct. Anal.}, 19(5):1468--1480, 2010.

\bibitem{mm2}
H.~A. Masur and Y.~N. Minsky.
\newblock Geometry of the complex of curves. {II}. {H}ierarchical structure.
\newblock {\em Geom. Funct. Anal.}, 10(4):902--974, 2000.

\bibitem{mm1}
Howard~A. Masur and Yair~N. Minsky.
\newblock Geometry of the complex of curves. {I}. {H}yperbolicity.
\newblock {\em Invent. Math.}, 138(1):103--149, 1999.

\bibitem{mccarthy}
John McCarthy.
\newblock A ``{T}its-alternative'' for subgroups of surface mapping class
  groups.
\newblock {\em Trans. Amer. Math. Soc.}, 291(2):583--612, 1985.

\bibitem{mcmullenkahler}
Curtis~T. McMullen.
\newblock The moduli space of {R}iemann surfaces is {K}\"ahler hyperbolic.
\newblock {\em Ann. of Math. (2)}, 151(1):327--357, 2000.

\bibitem{hypbyhyp}
Lee Mosher.
\newblock A hyperbolic-by-hyperbolic hyperbolic group.
\newblock {\em Proc. Amer. Math. Soc.}, 125(12):3447--3455, 1997.

\bibitem{rafi-metric}
Kasra Rafi.
\newblock A combinatorial model for the {T}eichm\"uller metric.
\newblock {\em Geom. Funct. Anal.}, 17(3):936--959, 2007.

\bibitem{rover}
Claas~E. R{\"o}ver.
\newblock On subgroups of the pentagon group.
\newblock {\em Math. Proc. R. Ir. Acad.}, 107(1):11--13 (electronic), 2007.

\bibitem{servatiusdroms}
Herman Servatius, Carl Droms, and Brigitte Servatius.
\newblock Surface subgroups of graph groups.
\newblock {\em Proc. Amer. Math. Soc.}, 106(3):573--578, 1989.

\bibitem{wang}
Stephen Wang.
\newblock Representations of surface groups and right-angled {A}rtin groups in
  higher rank.
\newblock {\em Algebr. Geom. Topol.}, 7:1099--1117, 2007.

\bibitem{yeung}
Sai-Kee Yeung.
\newblock Quasi-isometry of metrics on {T}eichm\"uller spaces.
\newblock {\em Int. Math. Res. Not.}, (4):239--255, 2005.

\end{thebibliography}
\bibliographystyle{plain}

\noindent
Department of Mathematics\\
Allegheny College\\\
Meadville, PA 16335\\
E-mail: \texttt{mclay@allegheny.edu}\\

\noindent
Department of Mathematics \\
University of Illinois at Urbana-Champaign \\
Urbana, IL 61801\\
E-mail: \texttt{clein@math.uiuc.edu}\\

\noindent
Department of Mathematics \\
University of Michigan\\
Ann Arbor, MI 48109\\
E-mail: \texttt{mangahas@umich.edu}

\end{document}